\documentclass[12pt]{amsart}
\usepackage{amssymb}
\usepackage{amsmath}
\usepackage{longtable}
\newcommand\g{{\mathfrak g}}
\newcommand\h{{\mathfrak h}}
\newcommand{\q}{\mathfrak{q}}
\newcommand\Bl{\operatorname{Bl}}
\newcommand\m{\mathfrak m}
\newcommand\n{\mathfrak n}
\newcommand\z{\mathfrak z}

\newcommand\p{\mathfrak p}

\newcommand\Sk{\mathcal S}
\newcommand\Der{\operatorname{Der}}
\newcommand\Leaf{\mathcal{L}}

\newcommand\codim{\operatorname{codim}}
\newcommand\Spec{\operatorname{Spec}}

\newcommand\F{\operatorname{F}}
\newcommand\W{{\bf A}}
\newcommand\K{\mathbb K}
\newcommand\U{\mathcal U}

\newcommand\D{\mathcal D}
\newcommand\Ann{\operatorname{Ann}}
\newcommand\Id{\mathfrak{Id}}

\newcommand\Walg{\mathcal W}
\newcommand\Z{\mathbb Z}
\newcommand\A{\mathcal A}
\newcommand\B{\mathcal B}
\newcommand\M{\mathcal M}

\newcommand\NN{\mathbb N}
\newcommand\N{\mathcal{N}}
\newcommand\gr{\operatorname{gr}}
\renewcommand\O{\mathbb O}
\newcommand\I{\mathcal I}
\newcommand\J{\mathcal J}
\renewcommand\sl{\mathfrak{sl}}
\newcommand\Sp{\mathop{\rm Sp}\nolimits}

\newcommand\Hom{\operatorname{Hom}}
\newcommand{\ad}{\mathop{\rm ad}\nolimits}
\newcommand{\Ad}{\mathop{\rm Ad}\nolimits}
\newcommand{\HC}{\operatorname{HC}}
\newcommand\Centr{\mathcal Z}
\newcommand\Goldie{\operatorname{Grk}}
\newcommand\mult{\operatorname{mult}}
\newcommand{\VA}{\operatorname{V}}

\newcommand\Nil{\mathcal{N}}
\newcommand\Orb{\mathbb{O}}

\newcommand\HVB{\operatorname{HVB}}
\newcommand\KF{\operatorname{K}}
\newcommand\rank{\operatorname{rk}}
\newcommand\LAnn{\operatorname{LAnn}}
\newcommand\Fun{\mathcal{F}}
\newcommand\Mod{\operatorname{Mod}}
\newcommand\Coh{\operatorname{Coh}}
\newtheorem{Thm}{Theorem}[subsection]
\newtheorem{Prop}[Thm]{Proposition}
\newtheorem{Cor}[Thm]{Corollary}
\newtheorem{Lem}[Thm]{Lemma}
\theoremstyle{definition}
\newtheorem{Ex}[Thm]{Example}
\newtheorem{defi}[Thm]{Definition}
\newtheorem{Rem}[Thm]{Remark}
\newtheorem{Conj}[Thm]{Conjecture}
\unitlength=1mm   \numberwithin{equation}{section}
\numberwithin{table}{section} \oddsidemargin=0cm
\author{Ivan  Losev}
\title{Finite dimensional representations of $W$-algebras}
\thanks{Address:  Department of Mathematics, Massachusetts Institute of
Technology, 77 Massachusetts Avenue, Cambridge, MA 02139, USA.
 E-mail: ivanlosev@math.mit.edu}
\begin{document}
\begin{abstract}
$W$-algebras of finite type are certain finitely generated
associative algebras closely related to the universal enveloping
algebras of semisimple Lie algebras. In this paper we prove a
conjecture of Premet that gives an almost complete classification of
finite dimensional irreducible modules for $W$-algebras. A key ingredient in our proof
is  a relationship between Harish-Chandra  bimodules and bimodules over $W$-algebras
that is also of independent interest.
\end{abstract}
\maketitle \tableofcontents
\section{Introduction}
\subsection{W-algebras}\label{SUBSECTION_W_intro}
Let $\g$ be a finite dimensional semisimple Lie algebra over an
algebraically closed field $\K$ of characteristic zero and $G$ be the simply connected algebraic group
with Lie algebra $\g$. Fix a
nilpotent element $e\in \g$ and let $\Orb$ denote its adjoint orbit.
Associated with the pair $(\g,e)$ is a certain associative unital
algebra $\Walg$ called the $W$-algebra (of finite type). In the
special case when $e$ is a principal nilpotent element this algebra
appeared in Kostant's paper \cite{Kostant}. In this case the
$W$-algebra is naturally isomorphic to the center $\Centr(\g)$ of
the universal enveloping algebra $\U:=U(\g)$. In the general case, a
definition of a W-algebra was given by Premet, \cite{Premet1}, 4.4.
Since then W-algebras were extensively studied, see, for instance,
\cite{BGK}, \cite{BK1}-\cite{BK3}, \cite{GG},\cite{Ginzburg},\cite{Wquant},
\cite{Premet2}-\cite{Premet4}.

Let us review Premet's definition
briefly. The definition is recalled in more detail in Subsection
\ref{SUBSECTION_Walg}.

To $e$ one assigns a certain subalgebra $\m\subset\g$ consisting of
nilpotent elements and of dimension $\frac{1}{2}\dim \Orb$, and also
a character $\chi:\m\rightarrow \K$. Set
$\m_\chi:=\{\xi-\langle\chi,\xi\rangle,\xi\in \m\}$. The $W$-algebra
$\Walg$ associated with the pair $(\g,e)$ is, by definition, the
quantum Hamiltonian reduction $(\U/\U\m_\chi)^{\ad\m}:=\{a+\U\m_\chi|
[\m,a]\subset \U\m_\chi\}$. This algebra has the following nice
features.

1) Choose an $\sl_2$-triple $(e,h,f)$ in $\g$ and set
$Q:=Z_G(e,h,f)$. There is an action of $Q$ on $\Walg$ by algebra
automorphisms. Moreover, there is a $Q$-equivariant embedding
$\q:=\operatorname{Lie}(Q)\hookrightarrow \Walg$ such that the
adjoint action of $\q\subset\Walg$ on $\Walg$ coincides with the
differential of the action of $Q$ on $\Walg$.

2) There is a distinguished increasing  exhaustive filtration (the Kazhdan filtration) $\KF_i\Walg, i\geqslant 0,$ of
$\Walg$ with $\KF_0\Walg=\K$. As Premet checked in \cite{Premet1},
the associated graded algebra is naturally identified with the
algebra of regular functions on a transverse slice
$S\subset\g$ to $\Orb$ called the {\it Slodowy slice}. The slice $S$ can
be defined as $e+\z_\g(f)$.

3) The space $\U/\U\m_\chi$ has a natural structure of a
$\U$-$\Walg$-bimodule. This allows to define the functor
$N\mapsto \Sk(N)$ from the category of (left)
$\Walg$-modules to the category of $\U$-modules: $\Sk(N):=
(\U/\U\m_\chi)\otimes_{\Walg}N$. This functor defines an
equivalence of $\Walg$-Mod with the full subcategory of $\U$-Mod consisting of all
{\it Whittaker} $\g$-modules, i.e., those, where the action of
$\m_\chi$ is locally nilpotent. The quasi-inverse functor is given
by $M\mapsto M^{\m_\chi}:=\{m\in M| \xi m=\langle\chi,\xi\rangle
m, \forall \xi\in \m\}$. This was proved by Skryabin in the appendix to
\cite{Premet1}.

\subsection{Finite dimensional irreducible representations}\label{SUBSECTION_fin_dim}

One of the  most important problems arising in representation theory
of associative algebras is to classify their irreducible finite
dimensional representations. Such representations are in one-to-one
correspondence with {\it primitive} ideals of finite codimension;
recall that a two-sided ideal is called  primitive if it
coincides with the annihilator of some irreducible module.

In \cite{Premet2} Premet proposed  to study the map $N\mapsto
\Ann_{\U}\Sk(N)$ from the set of all finite dimensional
irreducible $\Walg$-modules to the set of primitive ideals in $\U$.
He proved that the image consists of ideals, whose associated
variety in $\g$ coincides with $\overline{\Orb}$, and conjectured
that any such primitive ideal can be represented in the form
$\Ann_{\U}\Sk(N)$. This conjecture was proved by Premet in
\cite{Premet3}, Theorem 1.1,  under some mild restriction on an ideal, and   by the
author in \cite{Wquant} in the full generality,  alternative proofs
were recently found by Ginzburg, \cite{Ginzburg}, and Premet, \cite{Premet4}. Actually, the
author obtained a more precise result. He constructed two maps
$\I\mapsto \I^\dagger:\Id(\Walg)\rightarrow
\Id(\U), \J\mapsto \J_\dagger:
\Id(\U)\rightarrow \Id(\Walg)$ between the
sets $\Id(\Walg),\Id(\U)$ of two-sided ideals of $\Walg,\U$.
These two maps enjoy the following properties (see Theorem
\ref{Thm:5} for more details):
\begin{itemize}
\item[(a)] $\I^\dagger$ is primitive whenever $\I$ is. If, in addition, $\I$ is of finite codimension,
then the associated variety $\VA(\U/\I^\dagger)$ coincides with $\overline{\O}$.
\item[(b)] $\Ann_{\Walg}(N)^\dagger=\Ann_{\U}(\Sk(N))$ for any  $\Walg$-module $N$.
\item[(c)] $\codim_{\Walg}\J_\dagger=\operatorname{mult}_{\overline{\Orb}}(\U/\J)$ (see Subsection \ref{SUBSECTION_notation}
for the definition of $\operatorname{mult}_{\overline{\Orb}}$) provided $\VA(\U/\J)=\overline{\O}$.
\item[(d)] If $\J$ is primitive and $\VA(\U/\J)=\overline{\O}$, then $\{\I\in \Id_{fin}(\Walg)| \I^\dagger=\J\}$
coincides with the set of all primitive ideals of $\Walg$ containing
$\J_\dagger$.
\end{itemize}
Here and below
$\Id_{fin}(\Walg)$ denotes the set of all two-sided ideals of finite codimension in $\Walg$.

Premet suggested a stronger version of his existence conjecture
including also a uniqueness statement (e-mail correspondence).  The group $Q$ acts
naturally on $\Id_{fin}(\Walg)$. By 1) above, the unit component $Q^\circ$ of $Q$ acts on $\Id_{fin}(\Walg)$
trivially, so the action of $Q$ descends to that of the component group $C(e):=Q/Q^\circ$.

\begin{Conj}[Premet]\label{Conj:0}
For any primitive  $\J\in \Id_{\Orb}(\U):=\{\J\in \Id(\U)| \VA(\U/\J)=\overline{\O}\}$ the set of all
primitive ideals $\I\in \Id_{fin}(\Walg)$ with $\I^\dagger=\J$ is a
single $C(e)$-orbit.
\end{Conj}

Note that irreducible $\Walg$-modules, whose annihilators are
$C(e)$-conjugate, are very much alike. In the representation theory
of $\U$ there are (complicated) techniques allowing to describe the
set of primitive ideals in $\Id_{\Orb}(\U)$, see \cite{Jantzen} for details. So
Conjecture \ref{Conj:0} provides an almost complete classification
of irreducible finite dimensional representations of $\Walg$. This
classification is complete whenever the action of $C(e)$ on $\Id_{fin}(\Walg)$ is trivial. This is the
case, for example, when $\g=\sl_n$. Here $Q=Q^\circ Z(G)$ and $Z(G)$ acts trivially on $\Walg$. Here the classification was obtained by
Brundan and Kleshchev, \cite{BK2} by completely different methods (they used a
relation between $W$-algebras and shifted Yangians).

In Subsection \ref{SUBSECTION_finish} we derive Conjecture
\ref{Conj:0} from the following statement.

\begin{Thm}[Extended Premet's conjecture]\label{Thm:1}
An element $\I\in \Id_{fin}(\Walg)$ equals $\J_\dagger$ for some $\J\in \Id_{\Orb}(\U)$ if and only if
$\I$ is $C(e)$-invariant. If this is the case, then $\I=(\I^\dagger)_\dagger$.
\end{Thm}

\subsection{Harish-Chandra bimodules}\label{SUBSECTION_HC_intro}
Any two-sided ideal in $\U$ is a Harish-Chandra bimodule.
Recall that a $\U$-bimodule $\M$ is said to be Harish-Chandra if it is finitely
generated (as a bimodule) and the adjoint action of $\g$ on
$\M$ is locally finite meaning that every element of $\M$ is contained in a finite dimensional $\g$-submodule.
The Harish-Chandra $\U$-bimodules form an abelian category to be denoted by $\HC(\U)$.
We will see that, essentially, Theorem \ref{Thm:1} is a corollary of a more general
result on Harish-Chandra bimodules, see Theorem \ref{Thm:surjectivity}.

It turns out that there is a relationship between Harish-Chandra $\U$-bimodules and  $\Walg$-bimodules.
The idea to study this relationship was communicated to me by Ginzburg, his own approach is explained in
\cite{Ginzburg}.  Moreover, there is a hope, see Subsection \ref{SUBSECTION_further} for details, to obtain
the complete (not just modulo the $C(e)$-action) classification of finite dimensional irreducible
$\Walg$-modules studying the relationship between certain categories of bimodules.

It turns out that for a  $W$-algebra one can also define the notion of a Harish-Chandra bimodule.
A precise definition will be given in Subsection \ref{SUBSECTION_completions}. What we need to know right now is that any finite dimensional $\Walg$-bimodule is Harish-Chandra.
In fact,  we will consider the category of {\it $Q$-equivariant} Harish-Chandra $\Walg$-bimodules.
We say that a $\Walg$-bimodule $\Nil$ is {\it
$Q$-equivariant}, if it is equipped with a locally finite linear action of $Q$
such that
\begin{enumerate}
\item The structure map $\Walg\otimes \Nil\otimes \Walg\rightarrow\Nil$ is $Q$-equivariant.
\item The differential of the $Q$-action (defined since the action is locally finite) coincides with the adjoint action of $\q\subset\Walg$
on $\Nil$: $(\xi,n)\mapsto \xi n-n\xi$.
\end{enumerate}

$Q$-equivariant Harish-Chandra $\Walg$-bimodules form a
monoidal abelian category (tensor product  is the tensor product of $\Walg$-bimodules),
which we denote by $\HC^Q(\Walg)$. Inside  $\HC^Q(\Walg)$ there is the full subcategory
$\HC^Q_{fin}(\Walg)$ consisting of all finite dimensional $Q$-equivariant bimodules.

It turns out that the category
$\HC_{fin}^Q(\Walg)$ is closely related to a certain subquotient of $\HC(\U)$.

Namely,   consider the abelian category
$\HC_{\overline{\Orb}}(\U)$ of Harish-Chandra $\U$-bimodules $\M$
whose associated variety $\VA(\M)$ is contained  in $\overline{\Orb}$.
It has a Serre subcategory $\HC_{\partial{\Orb}}(\U)$ consisting of all
Harish-Chandra bimodules $\M$ with $\VA(\M)\subset \partial{\Orb}:=\overline{\Orb}\setminus \Orb$. We can form
the quotient category $\HC_{\Orb}(\U):=\HC_{\overline{\Orb}}(\U)/\HC_{\partial{\Orb}}(\U)$.

The category $\HC(\U)$  has a monoidal structure with respect to the tensor product of $\U$-bimodules. The subcategory $\HC_{\overline{\O}}(\U)$ is closed with respect to tensor products (but does not contain a unit of $\HC(\U)$). Clearly, the tensor product descends to $\HC_{\O}(\U)$.

In \cite{Ginzburg}, Section 4, Ginzburg constructed an exact functor
 $\HC_{\O}(\U)\rightarrow\HC_{fin}^Q(\Walg)$ (in fact, he did not consider
$Q$-actions but his construction can be easily upgraded to the
$Q$-equivariant setting, see Subsection \ref{SUBSECTION_Ginzburg}).
 Roughly speaking, this functor
 should be close to an equivalence (but there are strong evidences that it is not,
 the actual situation should be much subtler, see the next subsection).
In this paper we obtain some partial results towards this claim
to be stated now.

In Subsection \ref{SUBSECTION_construction2} we will construct
functors $\M\mapsto
\M_\dagger:\HC_{\overline{\Orb}}(\U)\rightarrow \HC_{fin}^Q(\Walg), \Nil\mapsto \Nil^\dagger: \HC_{fin}^Q(\Walg)\rightarrow
\HC_{\overline{\Orb}}(\U)$. The following theorem describes the
properties of these two functors.

\begin{Thm}\label{Thm:2}
\begin{enumerate}
\item The functor $\M\mapsto \M_\dagger$ is exact and left-adjoint to the functor   $\Nil\mapsto
\Nil^\dagger$. Moreover, for  $\J\in \Id_{\Orb}(\U)$ we have $(\U/\J)_\dagger=\Walg/\J_\dagger$.
\item Let $\M\in \HC_{\overline{\Orb}}(\U)$. Then
$\dim\M_{\dagger}=\mult_{\overline{\Orb}}(\M)$, and the kernel and the cokernel of the natural
homomorphism $\M\rightarrow (\M_\dagger)^\dagger$ lie in $\HC_{\partial\Orb}(\U)$.
\item $\M\rightarrow \M_{\dagger}$ is a tensor functor.
\item $\operatorname{LAnn}(\M)_{\dagger}=\operatorname{LAnn}(\M_\dagger), \operatorname{RAnn}(\M)_{\dagger}=\operatorname{RAnn}(\M_\dagger)$.
\item The functor $\M\mapsto \M_\dagger$ gives rise to an equivalence of
$\HC_{\Orb}(\U)$ and some full  subcategory in
$\HC_{fin}^Q(\Walg)$ closed under taking subquotients.
\end{enumerate}
\end{Thm}

Here $\operatorname{LAnn},\operatorname{RAnn}$ denote the left and right annihilators of a bimodule.

In fact, the functor $\bullet_\dagger$ mentioned in Theorem \ref{Thm:2} is obtained by the restriction
of an exact functor $\HC(\U)\rightarrow \HC^Q(\Walg)$ that also extends the map $\J\mapsto \J_\dagger$
between the sets of ideals. We will see in Subsection \ref{SUBSECTION_Ginzburg} that our functor $\bullet_{\dagger}$ essentially coincides with that of Ginzburg.

Finally, let us state a corollary of Theorem \ref{Thm:2} giving a sufficient condition
for semisimplicity of an object in $\HC_{\O}(\U)$. This corollary was suggested to the
author by R. Bezrukavnikov. It will be proved  in Subsection \ref{SUBSECTION_finish}.

\begin{Cor}\label{Cor:3}
Let $\M\in \HC_{\overline{\O}}(\U)$ be such that $\operatorname{LAnn}(\M),\operatorname{RAnn}(\M)$ are primitive
ideals. Then   $\M$ is semisimple in $\HC_{\Orb}(\U)$.
\end{Cor}

\subsection{Further developments}\label{SUBSECTION_further}
An important open problem about the functor $\bullet_\dagger: \HC_{\O}(\U)\rightarrow \HC^Q_{fin}(\Walg)$
is to describe its image.

Perhaps, the first question one can ask  is to describe irreducible objects in $\HC^Q_{fin}(\Walg)$ lying
in the image of $\bullet_{\dagger}$. Further, a reasonable restriction is to consider only $\Walg$-bimodules
with trivial  left and right central characters. Here we use a natural identification of $\Centr(\g)$ with  the center
of $\Walg$, see Subsection \ref{SUBSECTION_Walg} for details.
 A nonzero finite dimensional $\Walg$-module with trivial central character exists if and only if $\O$ is special in the sense of Lusztig.

In a subsequent paper we will prove the following result.
\begin{Thm}\label{Thm:new1} For any
finite dimensional irreducible $\Walg$-modules $N_1,N_2$ with trivial central character  there is
a simple object $\M\in \HC_\O(\U)$ such that $\Hom_\K(N_1,N_2)$ is a direct summand of $\M_\dagger$.
\end{Thm}

We remark that Theorem \ref{Thm:2} implies that $\M_\dagger$ is  simple in $\HC^Q_{fin}(\Walg)$
and so is semisimple as a $\Walg$-bimodule. Theorem \ref{Thm:new1} can be generalized to any
(possibly singular) integral central character but it fails for a non-integral one.

Theorem \ref{Thm:new1} establishes an interesting relationship between the
functor $\bullet_\dagger$ and the work of Bezrukavnikov, Finkelberg and Ostrik, \cite{BFO1},\cite{BFO2}.
In those papers  they proved (among other things) that the following two monoidal categories are "almost isomorphic":
\begin{itemize}
\item The subcategory $\HC^{ss}_{\O}(\U_0)$ of all semisimple objects in $\HC_\O(\U)$  with trivial left and right central characters; the claim that this category
is closed with respect to the tensor product functor follows from Corollary \ref{Cor:3}.
\item The category $\operatorname{Coh}^{A(\O)}(Y\times Y)$, where $A(\O)$ is the quotient
of $C(e)$ defined by Lusztig, see \cite{Lusztig}, and $Y$ is some finite set acted on by $A(\O)$; the notation
$\operatorname{Coh}^{A(\O)}$ means the category of all $A(\O)$-equivariant sheaves of finite dimensional vector spaces;
the tensor product on $\operatorname{Coh}^{A(\O)}(Y\times Y)$ is given by convolution.
\end{itemize}
"Almost isomorphic" means that the monoidal categories in consideration become isomorphic after modifying the associativity isomorphisms for triple tensor products in one of them, but actually in  all cases but few there is
a genuine isomorphism.

In a subsequent paper we will prove that for $Y$ we can take the set of (isomorphism classes of) irreducible finite dimensional
$\Walg$-modules with trivial central character (and so, in particular, that the $C(e)$-action on this
set factors through $A(\O)$). The first result in this direction is Theorem \ref{Thm:new1}.

Another question one can pose is whether the techniques of the present paper are specific
for the universal enveloping algebras of semisimple Lie algebras or they can be modified
to study analogous questions for other algebras. At the moment, we know one important class of algebras, where
this modification is possible: the symplectic reflection algebras (shortly, SRA) of Etingof and Ginzburg,
\cite{EG}. This is studied in our preprint \cite{SRA}.

Namely, let $V$ be a symplectic vector space and $\Gamma$ be a finite group acting on $V$
by linear symplectomorphisms. The symplectic reflection algebra $\mathcal{H}$ is a certain
flat deformation of the smash-product $SV\#\Gamma$, see \cite{EG} for a precise definition.
In the SRA setting the role of nilpotent orbits  is played by symplectic leaves of $V^*/\Gamma$.
The symplectic leaves are parametrized by conjugacy classes of  stabilizers  for the action of $\Gamma$
on $V$. Namely, let $\Leaf$ be a symplectic leaf. Pick a point $b\in V^*$, whose image in $V^*/\Gamma$
lies in $\Leaf$. Then to $\Leaf$ we assign the stabilizer $\underline{\Gamma}$ of $b$ in
$\Gamma$, the stabilizer is defined uniquely up to $\Gamma$-conjugacy.
The role of $Q$ is played by $N_\Gamma(\underline{\Gamma})$ and the role of $C(e)$ by $N_\Gamma(\underline{\Gamma})/\underline{\Gamma}$. 

It turns out that for SRA we have complete analogs of Theorems \ref{Thm:1},\ref{Thm:2}.
Both the constructions of maps and functors for SRA and the proofs that these maps and functors
have the required properties are very similar to (but more complicated technically than)
those of the present paper.

\subsection{Notation and conventions}\label{SUBSECTION_notation}
Let us explain several notions used below in the text.

{\bf Adjoint actions on bimodules.} Let $\A$ be an associative
algebra and $\M$ be an $\A$-bimodule. For $a\in \A$
we write $[a,m]:=am-ma$.

{\bf Associated varieties and multiplicities.} Let $\A$ be an
associative algebra equipped with an increasing filtration $\F_i\A$.
We suppose that $\gr\A:=\sum_{i\in \Z}\F_i\A/\F_{i-1}\A$ is a
Noetherian commutative algebra. Now let $\M$ be a filtered
$\A$-module such that $\gr\M$ is a finitely generated $\gr\A$-module. By
the {\it associated variety} $\VA(\M)$ of $\M$ we mean the support of $\gr\M$ in $\Spec(A)$.
Moreover, since $\gr\M$ is finitely generated, there is a $\gr\A$-module
filtration $\gr\M=M_0\supset M_1\supset M_2\supset\ldots\supset
M_k=\{0\}$ such that $M_i/M_{i+1}=(\gr\A)/\p_i$, where $\p_i$ is a
prime ideal in $\gr\A$. For an irreducible component $Y$ of
$\VA(\M)$ we write $\mult_Y\M$  for the number of indexes $i$ such that $\p_i$ coincides with the prime ideal
corresponding to $Y$. The number $\mult_Y\M$ is called the {\it
multiplicity} of $\M$ at $Y$. It is known that $\VA(\M)$ and
$\mult_Y\M$ do not depend on the choices we made.
When $\M$ is an $\A$-bimodule, $\VA(\M)$ stands for the associated variety of $\M$
regarded as a left $\A$-module.

Now let $\A_\hbar$ be an associative $\K[\hbar]$-algebra such that $\A_\hbar/(\hbar)$
is Noetherian and commutative. For a finitely generated $\A_\hbar$-module $\M_\hbar$
let $\VA(\M_\hbar)$ stand for the support of $\M_\hbar/\hbar \M_\hbar$ in $\Spec(\A_\hbar/(\hbar))$.
 For a component  $Y\subset\VA(\M_\hbar)$ one
sets $\mult_Y \M_\hbar:=\mult_Y (\M_\hbar/\hbar\M_\hbar)$.


{\bf Locally finite parts}. Let $\g$ be some Lie algebra and let $M$
be a module over $\g$. By the locally finite (shortly, l.f.) part of
$M$ we mean the sum of all finite dimensional $\g$-submodules of
$M$. Similarly, if $G$ is an algebraic group acting on $M$, then the $G$-locally finite part
of $M$ is the sum of all finite dimensional $G$-submodules, where the action of
$G$ is algebraic.

{\bf $\hbar$-saturated subspaces}. Let $V$ be a $\K[\hbar]$-module. We
say that a submodule $U\subset V$ is {\it $\hbar$-saturated} if $\hbar v\in U$
implies $v\in U$ for all $v\in V$.

{\bf Group actions on algebras, schemes, modules etc.} All algebraic group
actions considered in this paper are either algebraic or pro-algebraic. In the case of
algebras or modules "algebraic" means "locally finite" (see above). "Pro-algebraic" means
that an algebra is an inverse limit of locally finite $G$-modules. In particular,
given an action of an algebraic group $G$ on a scheme $X$, for any $\xi\in \g$ the velocity vector
field $\xi_*$ on $X$ is defined.

Let $A$ be an algebra equipped with an action of a group
$G$ by algebra automorphisms. By a $G$-equivariant $A$-module we mean
an $A$-module $M$ equipped with a $G$-action such that the structure map $A\otimes M\rightarrow M$
is $G$-equivariant.

By a $G$-{\it weakly} equivariant  $A$-bimodule (compare with the terminology used for $D$-modules)
we mean a $G$-equivariant $A\otimes A^{op}$-module. The notion of a {\it $G$-equivariant} bimodule is different,
compare with Subsection \ref{SUBSECTION_HC_intro}.


Below we gather some notation used in the paper.
\begin{longtable}{p{3cm} p{12cm}}
$\widehat{\otimes}$&the completed tensor product of complete topological vector spaces/ modules.\\
$(a_1,\ldots,a_k)$& the two-sided ideal in an associative algebra generated by  elements $a_1,\ldots,a_k$.\\
 $A^\wedge_\chi$&
the completion of a commutative algebra $A$ with respect to the maximal ideal
of a point $\chi\in \Spec(A)$.\\
$\Ann_\A(\M)$& the annihilator of an $\A$-module $\M$ in an algebra
$\A$.\\
$\Der(A)$& the Lie algebra of derivations of an algebra $A$.
\\
$G*_HV$&$:=(G\times V)/H$: the homogeneous vector bundle over $G/H$ with fiber $V$.\\
$g*_Hv$&the class of $(g,v)\in G\times V$ in $G*_HV$.\\
$G_x$& the stabilizer of $x$ in $G$.\\
$\Goldie(\A)$& the Goldie rank of a prime Noetherian algebra $\A$.\\
$\gr \A$& the associated graded vector space of a filtered
vector space $\A$.\\
$I(Y)$& the ideal in $\K[X]$ consisting of all functions vanishing
on $Y$ for a subvariety $Y$ in an affine variety $X$.\\
$\Id(\A)$& the set of all (two-sided) ideals of an algebra $\A$.\\
$\M_{\g-l.f.}$& the locally finite part of a
$\g$-module $\M$.\\ $R_\hbar(\A)$&$:=\bigoplus_{i\in
\mathbb{Z}}\hbar^i \F_i\A$ :the Rees $\K[\hbar]$-module of a filtered
vector space $\A$.
\\$U(\g)$& the universal enveloping algebra of a Lie algebra $\g$.
\\$\VA(\M)$& the associated variety of $\M$.\\
$\Centr(\g)$& the center of $U(\g)$.\\
$\Gamma(X,\mathcal{F})$& the space of global sections of a sheaf
$\mathcal{F}$ on $X$.
\end{longtable}

\subsection{Key ideas  and the content of the paper}
Below in this subsection we will briefly and somewhat informally  outline some key ideas and constructions of the paper. In the end of the subsection we will describe the structure of the paper.
\subsubsection{Auxiliary algebras and categories}
To construct the functor $\bullet_\dagger: \HC(\U)\rightarrow \HC^Q(\Walg)$
mentioned after Theorem \ref{Thm:2} we need certain intermediate algebras and their
categories of Harish-Chandra bimodules.

First of all, we need the {\it homogenizations}
$\U_\hbar,\Walg_\hbar$ of the algebras $\U,\Walg$ (to be defined formally in Subsections \ref{SUBSECTION_Walg},
\ref{SUBSECTION_decomp}). For example, $\U_\hbar$ is a graded $\K[\hbar]$-algebra such that $\hbar$ has degree 1
and $\U_\hbar/(\hbar-1)=\U$. In other words, $\U_\hbar$ is obtained from $\U$
by using the  Rees construction. Another way to view $\U_\hbar$ is as follows: as a graded
vector space $\U_\hbar$ is the same as $\K[\g^*][\hbar]$ but the product on $\U_\hbar$ is
a graded deformation of the usual commutative product.

One can introduce the notions of  Harish-Chandra bimodules
for $\U_\hbar,\Walg_\hbar$, Subsection \ref{SUBSECTION_completions}. Any Harish-Chandra bimodule, say for $\U_\hbar$, is obtained by using the Rees construction from a Harish-Chandra $\U$-bimodule equipped with a so called {\it good filtration}. The categories of Harish-Chandra $\U_\hbar$- and $\Walg_\hbar$-bimodules will be denoted
by $\HC(\U_\hbar),\HC(\Walg_\hbar)$.

Together with the algebras $\U_\hbar,\Walg_\hbar$ we will consider their completions
$\U_\hbar^\wedge,\Walg_\hbar^\wedge$ at the point $\chi:=(e,\cdot)\in \Orb$, where we consider $\Orb$
as an orbit in $\g^*$. Basically, the possibility
to complete at $\chi$ is a  reason why we need the homogenized algebras.

As well as for
commutative algebras, the completions are obtained by taking certain inverse limits
of the quotients of $\U_\hbar,\Walg_\hbar$ by the powers of appropriate maximal
ideals. But, informally, one can think that $\U_\hbar^\wedge$ is identified with
$\K[\g^*]^{\wedge}_\chi[\hbar]$. The product on $\U_\hbar^\wedge$ is a unique
continuous extension of the product from $\U_\hbar=\K[\g^*][\hbar]$
\footnote{Other kinds of completions
of $\U_\hbar,\Walg_\hbar$ appeared in the literature.
For example, Dodd and Kremnizer use some completions in \cite{DK}. However,
those are the $\hbar$-adic completions and their elements, in a sense,
are formal power series in $\hbar$ but are polynomial in the other variables. The completions
we use consist of formal power series in all variables and so are larger than those of
Dodd and Kremnizer.}.

In the sequel the algebras $\U_\hbar,\Walg_\hbar,\U^\wedge_\hbar,\Walg^\wedge_\hbar$ sometimes will
be called {\it quantum algebras}.

One of the main  reason  to consider  the completions $\U_\hbar^\wedge,\Walg^\wedge_\hbar$
is that they are related in a surprisingly simple way. Namely,  $\U_\hbar^\wedge$ is decomposed into the completed tensor product
of $\W_\hbar^\wedge$ and  $\Walg^\wedge_\hbar$, where $\W_\hbar^\wedge$ is the {\it completed homogeneous Weyl algebra} of an appropriate symplectic vector space. Basically, this is the "decomposition theorem" proved in
Subsection 3.3 of \cite{Wquant}. The theorem is useful because the algebra $\W^\wedge_\hbar$
is very simple and is an infinite dimensional analog of the matrix algebra (in the modular setting
an appropriate version of the universal enveloping also gets decomposed into
a similar tensor product but there one can take the genuine matrix algebra instead of
the Weyl algebra, this observation is a cornerstone of Premet's proof of the Kac-Weisfeiler
conjecture, see \cite{Premet1} for details).

One can introduce the notions
of Harish-Chandra bimodules for the algebras $\U_\hbar^\wedge,\Walg^\wedge_\hbar$.
This is done in such a way that the completion of a Harish-Chandra $\U_\hbar$- (resp. $\Walg_\hbar$-) bimodule
at $\chi$ is a Harish-Chandra $\U_\hbar^\wedge$-  (resp., $\Walg_\hbar^\wedge$-) bimodule.
The categories of Harish-Chandra $\U_\hbar^\wedge,\Walg_\hbar^\wedge$-bimodules will
be denoted by $\HC(\U_\hbar^\wedge),\HC(\Walg_\hbar^\wedge)$. These categories come equipped with the {\it completion functors}
$\HC(\U_\hbar)\rightarrow \HC(\U_\hbar^\wedge), \HC(\Walg_\hbar)\rightarrow \HC(\Walg_\hbar^\wedge)$.
As one expects, these functors are exact.

Now let us see what the decomposition theorem implies for
the categories $\HC(\Walg_\hbar^\wedge),\HC(\U_\hbar^\wedge)$. Above we made a speculative claim that
$\W_\hbar^\wedge$ is an infinite dimensional analog of a matrix algebra. So taking the
(completed) tensor product with $\W_\hbar^\wedge$ should give rise to a Morita equivalence.
This is not really so but, at least, we do get  an equivalence between $\HC(\U_\hbar^\wedge)$
and $\HC(\Walg_\hbar^\wedge)$ in this way, see  Proposition \ref{Prop:3.4.1} for an enhanced version
of this equivalence.

Now one can ask whether the completion functors are equivalences. It happens
that the answer is positive on the $\Walg$-side and negative on the $\U$-side.
The reason is that the algebra $\Walg_\hbar$ is positively graded, and the ideal
of $\chi$ is graded. So a quasi-inverse  to the completion functor is just
given by taking elements "of finite degree", see Proposition \ref{Prop:3.4.1}.

\subsubsection{A functor $\HC(\U_\hbar)\rightarrow\HC^Q(\Walg_\hbar)$}
Now we are ready to establish an exact functor $\HC(\U_\hbar)\rightarrow \HC(\Walg_\hbar)$.
The functor in interest is the composition of the completion functor $\HC(\U_\hbar)\rightarrow
\HC(\U_\hbar^\wedge)$ and the composition of the equivalences $\HC(\U_\hbar^\wedge)\rightarrow
\HC(\Walg_\hbar^\wedge)\rightarrow \HC(\Walg_\hbar)$.


As we mentioned before, the completion functor $\HC(\U_\hbar)\rightarrow \HC(\U_\hbar^\wedge)$
is not an equivalence. The argument we used for $\Walg_\hbar$ does not work here:
we still have some grading on $\U_\hbar$ such that the ideal of $\chi$ is graded
but this grading has both positive and negative components. There is a way to
improve the functor: it turns out that so far we have missed some structure
on the completion $\M_\hbar^\wedge$ of an object $\M_\hbar\in \HC(\U_\hbar)$.
Namely, the group $Q$ introduced in Subsection \ref{SUBSECTION_fin_dim} stabilizes $\chi$ and so acts naturally
on $\M_\hbar^\wedge$  making $\M_\hbar^\wedge$ a $Q$-equivariant Harish-Chandra
$\U^\wedge_\hbar$-bimodule in some precise sense, see Subsection \ref{SUBSECTION_completions}.
So we need to consider the category $\HC^Q(\U_\hbar^\wedge)$ of $Q$-equivariant Harish-Chandra
bimodules and also the similar categories $\HC^Q(\Walg^\wedge_\hbar),\HC^Q(\Walg_\hbar)$.
We remark that any object in $\HC(\U_\hbar)$ is also $Q$-equivariant but the action of
$Q$ is completely recovered from the adjoint action of $\g$ and so we do not need
to specify it.

As before, we have equivalences $\HC^Q(\Walg_\hbar)\cong \HC^Q(\Walg_\hbar^\wedge)\cong \HC^Q(\U_\hbar^\wedge)$, Proposition \ref{Prop:3.4.1}, and hence an exact functor $\HC(\U_\hbar)\rightarrow \HC^Q(\Walg_\hbar)$.

This functor is still not an equivalence. For example, it kills any Harish-Chandra
$\U_\hbar$-bimodule, whose associated variety does not contain $\chi$
(i.e., any bimodule that is zero in a neighborhood of $\chi$). However,
the functor in interest, and especially the induced functor
$\HC_\Orb(\U_\hbar)\rightarrow \HC_{fin}^Q(\Walg_\hbar)$, enjoys many properties
of an equivalence. To describe these properties it will be more convenient for us
to consider the completion functor $\HC(\U_\hbar)\rightarrow \HC(\U_\hbar^\wedge)$,
which differs from ours by an equivalence.

\subsubsection{A functor $\HC^Q(\U^\wedge_\hbar)\rightarrow \HC(\U_\hbar)$}\label{SSS_3}
For instance, there is a right adjoint functor $\HC^Q(\U^\wedge_\hbar)\rightarrow \HC(\U_\hbar)$.
This functor is produced in a pretty standard way.
Namely, given $\M_\hbar'\in \HC^Q(\U_\hbar^\wedge)$ we can take the subset $(\M_\hbar')_{l.f.}$
that is $\K^\times$- and $\g$-locally finite part of $\M_\hbar'$. Here we consider the modified version
of the adjoint $\g$-action: $(\xi,m)\mapsto \frac{1}{\hbar^2}(\xi m-m\xi)$.

Of course, the assignment $\M_\hbar'\mapsto
(\M_{\hbar}')_{l.f.}$ is functorial but this is not yet a functor we need.
The reason is that $(\M_\hbar')_{l.f.}$ comes equipped with two
different, in general, $Q$-actions such that $(\M_\hbar')_{l.f.}$ is an equivaraint bimodule with respect to both.
 A nice feature here
is that the two actions differ by an action of the component group $C(e):=Q/Q^\circ$
commuting with the $\U_\hbar\otimes \U_\hbar^{opp}$-action.
So to get a required functor we just  take the $C(e)$-invariants in $(\M_\hbar')_{l.f.}$.

The $\U_\hbar$-bimodule $(\M'_{\hbar})_{l.f.}^{C(e)}$  is actually finitely generated (this essentially
follows from results of Ginzburg, \cite{Ginzburg}) and  hence is Harish-Chandra.

In this paper we only need
a weaker statement: $(\M'_\hbar)_{l.f.}^{C(e)}$ is finitely generated whenever
$\M'_\hbar$ lies in the category $\HC^Q_{\Orb}(\U_\hbar^\wedge)$ of all objects
in $\HC^Q(\U_\hbar^\wedge)$ whose associated variety is contained in $\Orb$
(a more accurate way to say this would be "whose associated scheme is the formal
neighborhood of $\chi$ in $\Orb$"), see Lemma \ref{Lem:3.5.1}. A technique we use to prove this result
is also used in many other proofs of the paper so let us explain it briefly
here. A key idea is  to establish  similar properties
for  modules that are much simpler than $\M'_\hbar$ and than reduce the proof
for $\M'_\hbar$ to the proof for these modules.

Namely, as we will see, it is enough  to show that the $\K[\g^*]$-module $(\M'_{\hbar})_{l.f.}^{C(e)}/\hbar (\M'_{\hbar})_{l.f.}^{C(e)}$
is finitely generated. This module embeds into $(\M_{\hbar}'/\hbar \M_\hbar')^{C(e)}_{l.f.}$ (where $\M_{\hbar}'/\hbar \M_\hbar'$  is regarded as a $\K[\g^*]^{\wedge}_\chi$-module equipped with additional $Q$-,
$\g$- and $\K^\times$-actions that are compatible in some precise
way). Now, $\M_{\hbar}'/\hbar \M_\hbar'$ is supported on
$\Orb$. As such, it has a finite filtration, whose successive quotients are
$\K[\Orb]^\wedge_\chi$-modules. So we need to prove that for any finitely generated $Q$- and $\g$-equivariant $\K[\Orb]^{\wedge}_\chi$-module its submodule of $\g$-l.f.  elements is finitely generated over $\K[\overline{\Orb}]$. This statement is proved in three steps.
First of all, one shows that any finitely generated $Q$- and $\g$-equivariant $\K[\Orb]^\wedge_\chi$-module $N$
is, in fact, the restriction of a $G$-equivariant coherent sheaf $\mathcal{F}$ on   $\Orb$ to the formal
neighborhood of $\chi$. Next, one shows that the subspace of $C(e)$-invariant $\g$-locally finite elements
in $N$ coincides with the space $\Gamma(\Orb,\mathcal{F})$.
Finally, one uses the fact that the latter is a finitely generated $\K[\overline{\Orb}]$-module.

\subsubsection{The surjectivity property}\label{SSS_4}
Now let us comment on the most important  property of the completion functor
$\HC(\U_\hbar)\rightarrow \HC^Q(\U_\hbar^\wedge)$ ("surjectivity"), which is used to
prove Theorem \ref{Thm:1} and also the last part of Theorem \ref{Thm:2}.

  Suppose we are given a Harish-Chandra $\U_\hbar$-bimodule
$\M_\hbar$ and a $Q$-,$\K^\times$- and $\g$-stable sub-bimodule $\N'_\hbar$
of the completion $\M_\hbar^\wedge$ such that the quotient $\M^\wedge_\hbar/\N_\hbar$ is
supported on $\O$. Then there is a sub-bimodule in $\M_\hbar$, whose completion
coincides with $\N'_\hbar$. This claim is basically equivalent to Theorem \ref{Thm:surjectivity}.
In the proof we may assume that $\M_\hbar$ embeds into $\M_\hbar^\wedge$.

The proof roughly consists of two parts. First, we need to establish this property
in the case when $\N'_\hbar$ contains $\hbar^k \M^\wedge_\hbar$ for some $k$. Then
$\M^\wedge_\hbar/\N_\hbar$ has a filtration whose quotients are again finitely generated
$\K[\Orb]^\wedge_\chi$-modules. So we again need to use some properties of
coherent $G$-equivariant sheaves on $\Orb$. Second, we need to deduce the statement when  the
quotient $\M^\wedge_\hbar/\N'_\hbar$ is $\K[[\hbar]]$-flat.
Set $\N'_{\hbar,k}:=\N'_{\hbar}+\hbar^{k+1} \M^\wedge_{\hbar}, \N_{\hbar,k}:=\N'_{\hbar,k}\cap \M_\hbar$.
The first part of the proof guarantees that $\N_{\hbar,k}$ generates $\N'_{\hbar,k}$. What we need to show
is that the intersection $\bigcap_k \N_{\hbar,k}$ cannot be too small, i.e., that the sequence
$\N_{\hbar,k}$ stabilizes in some precise sense. This is done using Lemma \ref{Lem:4.3.2}.


\subsubsection{From the functors between $\HC(\U_\hbar),\HC(\Walg_\hbar)$ to  functors
between $\HC(\U),\HC^Q(\Walg)$}
Now let us explain how to pass from the functors between  $\HC(\U_\hbar)$ and $\HC^Q(\Walg_\hbar)$
to those between $\HC(\U)$ and $\HC^Q(\Walg)$. Taking the Rees bimodule gives rise to an equivalence
between the category of {\it filtered} Harish-Chandra $\U$-bimodules with good filtration and
$\HC(\U_\hbar)$. The similar statement holds for $\Walg$.
So we get functors $\bullet_{\dagger},\bullet^{\dagger}$ between the categories of  filtered
Harish-Chandra bimodules. However, it is not difficult to  see that the functors descend to the usual
categories of bimodules.

\subsubsection{The structure of the paper}
To finish the subsection let us explain how the paper is organized.

Section \ref{SECTION_prelim} contains some preliminary material. In its first subsection we
review basic properties of deformation quantization, the key
technique in the approach to W-algebras developed in \cite{Wquant}.
Also we recall some results on the existence of quantum (co)moment maps. In
Subsection \ref{SUBSECTION_Walg} we recall two definitions of
W-algebras: one due to Premet (in a variant of Gan-Ginzburg,
\cite{GG}) and one from \cite{Wquant}. Subsection \ref{SUBSECTION_decomp} recalls (and, in fact,
proves a stronger version of) the decomposition theorem. In Subsection \ref{SUBSECTION_bimod}
we prove some technical results on completions of quantum algebras. Finally in Subsection
\ref{SUBSECTION_completions} we discuss the notion of Harish-Chandra
bimodules for quantum algebras in interest.

Section \ref{SECTION_functors} is devoted mostly to constructing the
functors between the categories of bimodules. In Subsection
\ref{SUBSECTION_corr_ideals} we recall the definitions of the maps
$\I\mapsto\I^\dagger, \J\mapsto \J_\dagger$ from \cite{Wquant}.  Subsection
\ref{SUBSECTION_classical} is  technical, there we prove some
results on homogeneous vector bundles to be used both in the construction
of functors and in the proofs of the main theorems (roughly, as induction
steps, see the discussion above). Subsections \ref{SUBSECTION_Construction},\ref{SUBSECTION_construction2}
form a central part of the section: there we define the functors between
the categories of bimodules and study  basic properties of the
functors. At first, we do this on the level of quantum (=homogenized) algebras,
Subsection \ref{SUBSECTION_Construction}, and then on the level of
$\U,\Walg$, verifying  that our constructions essentially do not
depend on filtrations. In Subsection \ref{SUBSECTION_Ginzburg} we compare our construction
with Ginzburg's, \cite{Ginzburg}.

In Section \ref{SECTION_proofs} we complete the proofs of the two
main theorems. In the first subsection we  state some auxiliary
result (Theorem \ref{Thm:surjectivity}), which a straightforward generalization of
Theorem \ref{Thm:1} and also the most difficult part of Theorem \ref{Thm:2}.
Then in the second subsection we complete the proofs of Theorems \ref{Thm:1},\ref{Thm:2} using Theorem \ref{Thm:surjectivity}.

{\bf Acknowledgements.} The author is indebted to V. Ginzburg and A.
Premet for numerous stimulating discussions.

\section{Preliminaries}\label{SECTION_prelim}
\subsection{Deformation quantization and quantum comoment
maps}\label{SUBSECTION_defquant} Let $A$ be a commutative
associative $\K$-algebra with unit equipped with a Poisson bracket.

\begin{defi}\label{defi:1.1}
A map $*:A\otimes_{\K} A\rightarrow A[[\hbar]],
f*g=\sum_{i=0}^\infty D_i(f,g)\hbar^{2i}$ is called a {\it
star-product}   if it satisfies the
following conditions:
\begin{itemize}
\item[(*1)] The natural ($\K[[\hbar]]$-bilinear) extension of $*$ to $A[[\hbar]]\otimes_{\K[[\hbar]]}A[[\hbar]]$
is associative, i.e.,  $(f*g)*h=f*(g*h)$ for all $f,g,h\in
A$, and $1\in A$ is a unit for $*$.
\item[(*2)] $f*g-fg\in \hbar^2 A[[\hbar]],f*g-g*f-\hbar^2\{f,g\}\in \hbar^4 A[[\hbar]]$ for all $f,g\in A$
or, equivalently, $D_0(f,g)=fg, D_1(f,g)-D_1(g,f)=\{f,g\}$.
\end{itemize}
Star-products we deal with in this paper will satisfy the following
additional property.
\begin{itemize}
\item[(*3)] $D_i$ is a bidifferential operator of order at most $i$
in each variable ("bidifferential of order at most $i$" means that for any $f\in A$ both maps $g\mapsto D_i(g,f)$ and $g\mapsto D_i(f,g)$ are differential operators of order at most $i$).
\end{itemize}
\end{defi}
Note that usually a star-product is written as $f*g=\sum_{i=0}^\infty D_i(f,g)\hbar^i$ with
$f*g-g*f-\hbar\{f,g\}\in\hbar^2 A[[\hbar]]$. The reason why we use $\hbar^2$ instead of
$\hbar$ is that our choice is better compatible with the Rees algebra construction, which is used
to pass from a filtered $\K$-algebra to a graded $\K[\hbar]$-algebra. Namely, suppose
that we have an algebra $\A$ equipped with an increasing filtration $\F_i\A$. Then the Rees
$\K[\hbar]$-module $R_\hbar(\A)$  is naturally a graded $\K[\hbar]$-algebra. In most of our applications,
we will have $[\F_i\A,\F_j\A]\subset \F_{i+j-2}\A$ for all $i,j$. This condition translates to
$[a,b]\in \hbar^2 R_\hbar(\A)$ for any $a,b\in R_\hbar(\A)$, which is similar to (*2).

When we consider $A[[\hbar]]$ as an
algebra with respect to the star-product, we call it a {\it quantum algebra}.
If $A[\hbar]$ is a subalgebra in $A[[\hbar]]$ with respect to $*$, then we
say that $*$ is a {\it polynomial} star-product, $A[\hbar]$ is also
called a quantum algebra.

\begin{Ex}[The Weyl algebra $\W_\hbar$]\label{Ex:1.1}
Let $X=V$ be a finite-dimensional vector space equipped with a
constant nondegenerate Poisson bivector $P$.  The {\it Moyal-Weyl}
star-product on $A:=\K[V]$  is defined by
$$f*g=\exp(\frac{\hbar^2}{2}P)f(x)\otimes g(y)|_{x=y}.$$
Here $P$ is considered as an element of $V\otimes V$. This space
acts naturally on $\K[V]\otimes \K[V]$ (by contractions). The
quantum algebra $\W_\hbar:=\K[V][\hbar]$ is called the {\it
(homogeneous) Weyl algebra}.
\end{Ex}


Now we discuss group actions on quantum algebras.

Let  $G$ be an algebraic group acting on $A$ by automorphisms. It
makes sense to speak about  $G$-invariant star-products ($\hbar$ is
supposed to be $G$-invariant).


Recall that a $G$-equivariant linear map $\xi\mapsto H_\xi:
\g:=\operatorname{Lie}(G)\rightarrow A$ is said to be a {\it
comoment map} if $\{H_\xi, \bullet\}=\xi_*$ for any $\xi\in \g$.
The action  of $G$ on $A$  equipped with a comoment map is called {\it
Hamiltonian}. In the case when $A$ is finitely generated define the
{\it moment map} $\mu:\Spec(A)\rightarrow \g^*$ to be the dual map
to the comoment map $\g\mapsto A$. We remark that for a given action of
$G$ a comoment map is, in general, not unique. However, for two comoment
maps $\xi\mapsto H^1_\xi,H^2_\xi$ the element $H^1_\xi-H^2_\xi$ Poisson commutes with
$A$. In particular, if the Poisson center of $A$ consists of scalars (for example,
if $A$ is an algebra of functions on a smooth affine symplectic variety) then
two comoment maps $H^1_\bullet,H^2_\bullet$ differ by a $G$-invariant element of $\g^*$,  i.e.,
there is $\alpha\in (\g^*)^G$ with $H^1_\xi-H^2_\xi=\langle\alpha,\xi\rangle$
 for any $\xi\in \g^*$.

In the quantum situation there is an analog of a  comoment map
defined as follows: a $G$-equivariant linear map $\g\rightarrow
A[[\hbar]], \xi\mapsto \widehat{H}_\xi,$ is  said to be a {\it quantum comoment map} if
$[\widehat{H}_\xi,\bullet]=\hbar^2 \xi_*$ for all $\xi\in\g$. If the Poisson center
of $A$ consists of scalars, then two different quantum comoment maps differ by
an element of $(\g^*)^G[[\hbar]]$, compare with the previous paragraph.

Now let $\K^\times$ act on $A, (t,a)\mapsto t.a$, by automorphisms.
For instance, if $A$ is graded, $A:=\bigoplus_{i\in\Z} A_i$, we can consider
the action coming from the grading: $t.a=t^i a, a\in A_i$.
Consider the action of $\K^\times$ on $A[[\hbar]]$ given by
$t.\sum_{j=0}^\infty a_j\hbar^j=\sum_{j=0}^\infty
t^{j}(t.a_j)\hbar^j$. If $\K^\times$ acts by automorphisms of $*$,
then we say that $*$ is {\it homogeneous}. Clearly, $*$ is homogeneous if and only if
the map $D_l:A\otimes A\rightarrow A$ is homogeneous of degree $-2l$.

The following theorem on existence of star-products and quantum
comoment maps incorporates results of Fedosov,
\cite{Fedosov1}-\cite{Fedosov3}, in the form we need.

\begin{Thm}\label{Thm:3}
Let $X$ be a smooth affine  variety  equipped with
\begin{itemize}\item a symplectic form $\omega$, \item a Hamiltonian action of
a reductive group $G$, $\xi\mapsto H_\xi$ being a comoment map,
\item and an action of the one-dimensional torus $\K^\times$ by
$G$-equivariant automorphisms  such that $t.\omega=t^2\omega,
t.H_\xi=t^2H_\xi$.\end{itemize} Then there exists a $G$-invariant
homogeneous star-product $*$ on $\K[X]$ satisfying the additional
condition (*3) and such that $\xi\mapsto H_\xi$ is a quantum
comoment map.
\end{Thm}

For instance, in Example \ref{Ex:1.1},  $*$ satisfies the
conditions of Theorem \ref{Thm:3} with $G=\Sp(V)$ and the action of $\K^\times$ given by
$t.v=t^{-1}v$. Note that $H_\xi(v)=\frac{1}{2}\omega(\xi v,v), \xi \in \mathfrak{sp}(V), v\in V$.

\begin{Rem}\label{Rem:3}
Actually, for a Fedosov star-product $*$ one has $D_i(f,g)=(-1)^iD_i(g,f)$. This is proved, for instance,
in \cite{BW}, Lemma 3.3. In particular, $D_1(f,g)=\frac{1}{2}\{f,g\}$ and $D_2$ is symmetric.
\end{Rem}

Fedosov proved an analog of Theorem  \ref{Thm:3}
in the $C^\infty$-setting. The explanation why  his results
(except that on a quantum comoment map) hold in the algebraic setting together with references can be found in \cite{Wquant}, Subsection 2.2. The statement on the quantum comoment  map is Theorem 2 in \cite{Fedosov3}.
Its proof can be transferred to the algebraic setting directly.


Note also that, since $*$ is $G$-invariant, we get a well-defined
star-product on $\K[X]^G$. In this way, taking $X=T^*G$ and
replacing $G$ with $G\times G$, one gets a $G$-invariant star-product on
$S(\g)=\K[\g^*]$. The corresponding quantum algebra will be denoted
by $\U_\hbar$. This notation is justified by the observation that
$\U_\hbar/(\hbar-1)\cong \U$, see \cite{Wquant}, Example 2.2.4, for
details. We will encounter another example of this construction in
the following subsection.

\subsection{W-algebras}\label{SUBSECTION_Walg}
In this subsection we review the definitions of W-algebras due to
Premet, \cite{Premet1}, and the author, \cite{Wquant}.

Recall that  a nilpotent element $e\in \g$ is fixed and  $G$ denotes the simply connected algebraic group
with Lie algebra $\g$, $\Orb:=Ge$. Choose an
$\sl_2$-triple $(e,h,f)$ in $\g$ and set $Q:=Z_G(e,h,f)$. Also
introduce a grading on $\g$ by eigenvalues of $\ad h$:
$\g:=\bigoplus_{i\in \Z} \g(i), \g(i):=\{\xi\in\g| [h,\xi]=i\xi\}$. Since $h$ is the image of a coroot
under a Lie algebra homomorphism $\sl_2\rightarrow \g$, we see that there is a unique one-parameter
subgroup $\gamma:\K^\times\rightarrow G$ with $d_1\gamma(1)=h$.

The Killing form $(\cdot,\cdot)$ on $\g$ allows to identify
$\g\cong\g^*$, let $\chi=(e,\bullet)$ be an element of $\g^*$
corresponding  to $e$. Identify $\Orb$ with $G\chi$. Note that
$\chi$ defines a symplectic form $\omega_\chi$ on $\g(-1)$ as
follows: $\omega_\chi(\xi,\eta)=\langle \chi,[\xi,\eta]\rangle$. Fix
a lagrangian subspace $l\subset \g(-1)$ with respect to $\omega_\chi$ and set
$\m:=l\oplus\bigoplus_{i\leqslant -2} \g(i)$. Define the affine
subspace $\m_\chi\subset\g$ as in Subsection \ref{SUBSECTION_W_intro}. Then, by definition, the
$W$-algebra $\Walg$ associated with $(\g,e)$ is $(\U/\U\m_\chi)^{\ad \m}:=\{a+ U\m_\chi| [\m,a]\subset \U\m_\chi\}$.

Let us introduce a  filtration on $\Walg$. We have the standard PBW
filtration on $\U$ (by the order of a monomial) denoted by $\F^{st}_i
\U$. The {\it Kazhdan filtration}  $\KF_i\U$ is defined by $\KF_i
\U:=\sum_{2k+j\leqslant i}\F^{st}_k\U\cap \U(j)$, where $\U(j)$ is the
eigenspace of $\ad h$ on $\U$ with eigenvalue $j$. Note that the associated graded
algebra of $\U$ with respect to the Kazhdan filtration is still naturally isomorphic
to the symmetric algebra $S(\g)$. Being a subquotient of $\U$, the
$W$-algebra $\Walg$ inherits the Kazhdan filtration (denoted by $\KF_i
\Walg$). It is easy to see that $\KF_0\U\subset \U\m_\chi+\K$ and hence $\KF_0\Walg=\K$.

There are two disadvantages of this definition of $\Walg$. First, formally it
depends on a choice of $l\subset\g(-1)$. Second,  one cannot see
an action of $Q$ on $\Walg$ from it. Both disadvantages are remedied
by a ramification of Premet's definition given by Gan and Ginzburg
in \cite{GG}. Namely, they checked that there is a natural
isomorphism $(\U/\U\g_{\leqslant -2,\chi})^{\ad \g_{\leqslant
-1}}\rightarrow \Walg$, where $\g_{\leqslant k}:=\sum_{i\leqslant
k}\g(i),\g_{\leqslant -2,\chi}:=\{\xi-\langle\chi,\xi\rangle| \xi\in
\g_{\leqslant -2}\}$. Since all $\g_{\leqslant k}$ are $Q$-stable,
the group $Q$ acts naturally on $(\U/\U\g_{\leqslant
-2,\chi})^{\ad\g_{\leqslant -1}}$ and it is clear that the action is by
algebra automorphisms. Also Premet checked in \cite{Premet2} that
there is an inclusion $\q\hookrightarrow \Walg$ compatible with the
action of $Q$ in the sense explained in the Introduction.

Finally, note that there is a natural homomorphism
$\Centr(\g)\hookrightarrow \U^{\ad\m}\rightarrow (\U/\U\m_\chi)^{\ad\m}$.
Premet checked in \cite{Premet2} that it is injective and identifies
$\Centr(\g)$ with the center of $\Walg$.

Now let us recall the definition of $\Walg$ given in \cite{Wquant}.
Define the Slodowy slice $S:=e+\z_\g(f)$. It will be convenient for
us to consider $S$ as a subvariety in $\g^*$ via the identification $\g\cong\g^*$.  Define the {\it
Kazhdan} action of $\K^\times$ on $\g^*$ by
$t.\alpha=t^{-2}\gamma(t)\alpha$ for $\alpha\in
\g^*(i):=\g(i)^*$. This action preserves $S$ and, moreover,
$\lim_{t\rightarrow \infty}t.s=\chi$ for all $s\in S$. Define the
{\it equivariant Slodowy slice} $X:=G\times S$. The variety $X$ is naturally embedded into
$T^*G=G\times\g^*$. Here we use the trivialization of $T^*G$ by left-invariant 1-forms
to identify $T^*G$ with $G\times\g^*$. Equip $T^*G$ with a $\K^\times$-action given
by $t.(g,\alpha)=(g\gamma(t)^{-1},t^{-2}\gamma(t)\alpha)$ and with
a $Q$-action by $q.(g,\alpha)=(gq^{-1}, q\alpha), q\in Q,g\in
G,\alpha\in\g^*$. The equivariant Slodowy slice is stable under both
actions. The action of $G\times Q$ on $T^*G$ (and hence on $X$) is
Hamiltonian with a moment map $\mu$ given by
$\langle\mu(g,\alpha), (\xi,\eta)\rangle=\langle
\Ad(g)\alpha,\xi\rangle+\langle\alpha,\eta\rangle,
\xi\in\g,\eta\in\q$.

According to \cite{Wquant}, Subsection 3.1,   the Fedosov star-product on $\K[X]$ is polynomial.
So we have a quantum algebra $\widetilde{\Walg}_\hbar:=\K[X][\hbar]$ (the equivariant W-algebra).

Taking the $G$-invariants in $\widetilde{\Walg}_\hbar$, we get a
homogeneous $Q$-equivariant star-product  $*$ on $\Walg_\hbar:=\K[S][\hbar]$ together
with a quantum comoment map $\q\rightarrow \K[S]$. This map is injective because
the action of $Q$ on $X$ is free. 
Note that the grading
on $\Walg_\hbar$ induces a filtration on $\Walg_\hbar/(\hbar-1)$.
Also note that the quantum comoment map $\g\rightarrow \widetilde{\Walg}_\hbar$
gives rise to a homomorphism
$\U_\hbar^G\rightarrow \Walg_\hbar$.  So we get a homomorphism
$\Centr(\g)\hookrightarrow \Walg_\hbar/(\hbar-1)$. It is not difficult to check
that this homomorphism is an isomorphism of $\Centr(\g)$ onto  the center of $\Walg$ but we
will not give a direct proof of this.

\begin{Thm}\label{Thm:4}
There is a $Q$-equivariant isomorphism $\Walg_\hbar/(\hbar-1)\cong
\Walg$ of filtered algebras intertwining the homomorphisms from
$\Centr(\g)$.
\end{Thm}

Almost for sure, one can assume, in addition, that an isomorphism intertwines also
the embeddings of $\q$. However, a slightly weaker claim follows from the $Q$-equivariance:
namely, that the embeddings of $\q$ into $\Walg\cong \Walg_\hbar/(\hbar-1)$ differ by a character of $\q$.

A  version of this theorem, which did not take
the embeddings of $\Centr(\g)$ into account, was proved
in \cite{Wquant}, Corollary 3.3.3.  To see that this isomorphism
intertwines the maps from $\Centr(\g)$ one can argue as follows. According to \cite{Wquant}, Remark 3.1.3,
we have a $G$-equivariant isomorphism of $\widetilde{\Walg}_\hbar/(\hbar-1)$ and the algebra $\D(G,e):=(\D(G)/\D(G)\m_\chi)^{\ad \m}$, where $\m$ is embedded into $\D(G)$ via the velocity
vector field map for the right $G$-action. For both algebras we have quantum comoment maps $\g\rightarrow \widetilde{\Walg}_\hbar/(\hbar-1), \D(G,e)$.  This $G$-equivariant isomorphism
intertwines the quantum comoment maps because $G$ is semisimple. The isomorphism in Theorem \ref{Thm:4} is obtained
by restricting this $G$-equivariant isomorphism to the $G$-invariants. So the maps from $\Centr(\g)$ are indeed intertwined.

\subsection{Decomposition theorem}\label{SUBSECTION_decomp}
In a sense, the decomposition theorem is a basic result about  W-algebras. In a sentence, it says
that, up to completions, the universal enveloping algebra is decomposed into the tensor product
of the W-algebra and of a Weyl algebra. We start with an equivariant version of this theorem.

Apply Theorem \ref{Thm:3} to $X,T^*G$. We get a $G\times Q$-equivariant
homogeneous star-product $*$ on $X$ and $T^*G$ together with  quantum comoment
maps $\g\times \q\rightarrow \K[X][[\hbar]],\K[T^*G][[\hbar]]$.

Since the star-products on both $\K[X][[\hbar]]$ and $\K[T^*G][[\hbar]]$ are differential,
we can extend them to the completions $\K[X]^\wedge_{Gx}[[\hbar]], \K[T^*G]^\wedge_{Gx}[[\hbar]]$
along the $G$-orbit $Gx$, where $x=(1,\chi)$. These algebras come equipped with natural topologies
-- the topologies of completions.

Set $V:=[\g,f]$. Equip $V$ with the symplectic form
$\omega(\xi,\eta)=\langle\chi,[\xi,\eta]\rangle$, the action of
$\K^\times: t.v=\gamma(t)^{-1}v$ and the natural action of $Q$.

\begin{Thm}\label{Thm_decomp}
There is a $G\times Q\times \K^\times$-equivariant $\K[[\hbar]]$-linear isomorphism $$\Phi_\hbar:\K[T^*G]^\wedge_{Gx}[[\hbar]]
\rightarrow \K[X]^\wedge_{Gx}[[\hbar]]\widehat{\otimes}_{\K[[\hbar]]}\K[V^*]^\wedge_0[[\hbar]]$$ intertwining
the quantum comoment maps from $\g\times\q$.
\end{Thm}
Here $\widehat{\otimes}$ stands for the completed tensor product
(i.e., we consider the tensor product of the topological algebras $\K[X]^\wedge_{Gx}[[\hbar]],\K[V^*]^\wedge_0[[\hbar]]$ and then complete this tensor product
with respect to the induced topology).
\begin{proof}
Recall from the proof of Theorem 3.3.1 in \cite{Wquant} that there is a $G\times Q\times \K^\times$-equivariant isomorphism $\varphi:\K[T^*G]^\wedge_{Gx}\rightarrow \K[X]^\wedge_{Gx}\widehat{\otimes}\K[V^*]^\wedge_0$
of Poisson algebras. We remark that the centers of the algebras in interest consist of scalars.  The automorphism
 $\varphi$ is $G\times Q$-equivariant so it intertwines the classical comoment
maps perhaps  up to  a character of $\g\times \q$, see the explanation on the difference between
two classical comoment maps in Subsection \ref{SUBSECTION_defquant}. Let us remark that
the functions $H_\xi$ have degree 2 with respect to the $\K^\times$-action. Since $\varphi$ is  $\K^\times$-equivariant,
 and the centers of our algebras consist of scalars, we see that
the character has to vanish. Identify $A:=\K[T^*G]^\wedge_{Gx}$ and $\K[X]^\wedge_{Gx}\widehat{\otimes}\K[V^*]^\wedge_0$
by means of $\varphi$. Let $*,*'$ denote
the two star-products on $A[[\hbar]]$. Then there is  $\Phi_\hbar=\operatorname{id}+\sum_{i=1}^\infty T_i\hbar^{2i}$, where $T_i$ is a $G\times Q$-invariant differential operator of degree $-2i$ with respect to $\K^\times$, with $\Phi_\hbar(f)*'\Phi_\hbar(g)=\Phi_\hbar(f*g)$, see Theorem 3.3.1 in \cite{Wquant}. We need to check that $\Phi_\hbar(H_\xi)=H_\xi$.

We claim that $T_1$ is a derivation of $A$.  To prove this consider the equality $\Phi_\hbar(f)*'\Phi_{\hbar}(g)=\Phi_\hbar(f*g), f,g\in A$, modulo $\hbar^4$. We have
\begin{align*}
\Phi_\hbar(f)*'\Phi_{\hbar}(g)&\equiv (f+T_1(f)\hbar^2)*'(g+T_1(g)\hbar^2)\equiv\\
&\equiv fg+ (fT_1(g)+T_1(f)g)\hbar^2+
+\frac{1}{2}\{f,g\}\hbar^2\mod \hbar^4,\\ \Phi_\hbar(f*g)&\equiv f*g+ T_1(f*g)\equiv fg+ \frac{1}{2}\{f,g\}\hbar^2+ T_1(fg)\hbar^2\mod \hbar^4.
\end{align*}
These two congruences imply the claim. Now consider the skew-symmetric
parts of the coefficients of $\hbar^4$. Using the fact that $D_2$ is symmetric, see  Remark \ref{Rem:3},
we see that this part in $\Phi(f)*'\Phi(g)$ equals $\{T_1(f),g\}\hbar^4+ \{f,T_1(g)\}\hbar^4$, while
for $\Phi_\hbar(f*g)$ we have $T_1(\{f,g\})\hbar^4$.
Comparing the two, we deduce that $T_1$ annihilates the Poisson bracket.

Let $\eta$ be the vector field on $(T^*G)^\wedge_{Gx}$ corresponding to $T_1$. Since $H^1_{DR}(T^*G^\wedge_{Gx})=H^1_{DR}(G)=\{0\}$, we see that $\eta$ is a Hamiltonian vector field.
So $T_1=\{f,\cdot\}$ for some $G\times Q$-equivariant function $f$. So $T_1(H_\xi)=\{f,H_\xi\}=0$ for any $\xi\in\g\times\q$. In other words, $\Phi_\hbar(H_\xi)-H_\xi\in \hbar^4 A$. On the other hand,
$\Phi_\hbar(H_\xi)-H_{\xi}$ is a central element (hence lies in $\K[[\hbar]]$) and has degree 2 with respect to the $\K^\times$-action. So this element is zero.
\end{proof}

\begin{Rem}
In the proof of the theorem we used the semisimplicity of $G$. However, Theorem \ref{Thm_decomp} holds  for a general
reductive group too (we will need this situation in a subsequent paper).
To show this suppose, at first, that  $G=Z\times G_0$, where $Z:=Z(G)^\circ, G_0:=(G,G)$. Then $T^*G=T^*Z\times T^*G_0$ and $X=T^*Z\times X_0$, where $X_0$ is the equivariant Slodowy slice for $G_0$. Also we have the decomposition
$Q=Z\times Q_0, Q_0:=Q\cap G_0$. Choose a $Z\times  \K^\times$-invariant symplectic connection
on $T^*Z$, and $G_0\times Q_0\times \K^\times$-invariant symplectic connections on $T^*(G,G)$ and $X_0$.
Equip $T^*G$ and $X$ with the products of the corresponding connections. Then  for $\Phi_\hbar$ we can take
the product of the identity isomorphism $\K[T^*Z]^\wedge_{Z}[[\hbar]]\rightarrow \K[T^*Z]^\wedge_Z[[\hbar]]$
and an isomorphism
$$\K[T^*G_0]^\wedge_{G_0x}[[\hbar]]
\rightarrow \K[X_0]^\wedge_{G_0x}[[\hbar]]\widehat{\otimes}_{\K[[\hbar]]}\K[V^*]^\wedge_0[[\hbar]],$$
satisfying the conditions of the theorem.

In general, to get the decomposition $G=Z\times G_0$ one needs to replace $G$ with a covering. So let
$G=(Z\times G_0)/\Gamma$, where $\Gamma$ is a finite central subgroup in $Z\times G_0$. An isomorphism $\Phi_\hbar$ from the previous paragraph is $\Gamma$-equivariant. Its restriction
to $\Gamma$-invariants has the required properties.
\end{Rem}

 Consider the quantum algebras $\U_{\hbar}:=\K[\g^*][\hbar]=\K[T^*G][\hbar]^G,
\U_\hbar^\wedge:=\K[\g^*]^\wedge_\chi[[\hbar]],
\W_\hbar^\wedge:=\K[V^*]^\wedge_0[[\hbar]],
\Walg_\hbar:=\K[S][\hbar],
\Walg_\hbar^\wedge:=\K[S]^\wedge_\chi[[\hbar]]$ and finally
$\W_\hbar^\wedge(\Walg^\wedge_\hbar):=\W_\hbar^\wedge\widehat{\otimes}_{\K[[\hbar]]}\Walg^\wedge_\hbar$.

Restricting $\Phi_\hbar$  from Theorem \ref{Thm_decomp} to the $G$-invariants, we get a
$Q$-and $\K^\times$-equivariant isomorphism $\Phi_\hbar:
\U^\wedge_\hbar\rightarrow \W_\hbar^\wedge(\Walg^\wedge_\hbar)$ of
topological $\K[[\hbar]]$-algebras. Till the end of the paper we fix this
isomorphism.

\subsection{Completions of quantum algebras}\label{SUBSECTION_bimod}
Let $A$ be a finitely generated Poisson algebra, and $\A_\hbar=
A[\hbar]$ be a quantum algebra with a star-product
$f*g=\sum_{i=0}^\infty D_i(f,g)\hbar^{2i}$ satisfying the condition
(*3). Suppose that $A$ is graded, $A=\bigoplus_{i\in \Z}A_i$ with $\{A_i,A_j\}\subset
A_{i+j-2}$. Further, suppose that $*$ is homogeneous.

Choose a $\K^\times$-invariant point $\chi\in \Spec(A)$. We have the natural
structure of a quantum algebra on
$\A^\wedge_\hbar:=A^\wedge_\chi[[\hbar]]$.
We will be particularly interested in $\A_\hbar=\U_\hbar,\Walg_\hbar$
(equipped with the Kazhdan $\K^\times$-actions) with $\chi=(e,\cdot)$.

Let $I^\wedge_{\chi,\hbar}$ denote the inverse image of the maximal
ideal $\m_\chi\subset A^\wedge_\chi$ of $\chi$ in $\A^\wedge_\hbar$
and $I_{\chi,\hbar}:=I^\wedge_{\chi,\hbar}\cap \A_\hbar$. Then
$I^\wedge_{\chi,\hbar}, I_{\chi,\hbar}$ are two-sided ideals  of the
corresponding quantum algebras and their powers $(I^\wedge_{\chi,\hbar})^m, I_{\chi,\hbar}^m$ with respect to
the star-products coincide with  the powers with respect to the commutative products.
The last claim follows easily from  (*3). Now it is very easy to see
that $\A^\wedge_\hbar$ is naturally isomorphic to the completion
$\varprojlim_k \A_\hbar/I_{\chi,\hbar}^k$. If a group $Q$ acts on $A$
preserving $\chi$ and $*$, then we have a natural action of $Q$ on
$\A^\wedge_\hbar$.

Let $\M_\hbar$ be a  finitely generated
$\A_\hbar$-module. To $\M_\hbar$ one can assign its completion
$\M^\wedge_\hbar:=\varprojlim \M_\hbar/ I_{\chi,\hbar}^k\M_\hbar$,
which has a natural structure of an
$\A^\wedge_\hbar$-module. If $\M_\hbar$ is $\K^\times$-equivariant, then so is
$\M^\wedge_\hbar$.

\begin{Prop}\label{Lem:3.0.1}
\begin{enumerate}
\item
$\M^\wedge_\hbar =\A^\wedge_\hbar\otimes_{\A_\hbar}\M_\hbar$ and the
functor $\M_\hbar\mapsto \M^\wedge_\hbar$ is exact.
\item $\M^\wedge_\hbar=0$ if and only if $\chi\not\in\VA(\M_\hbar)$.
\item If $\M_\hbar$ is $\K[\hbar]$-flat, then $\M^\wedge_\hbar$ is
$\K[[\hbar]]$-flat.
\item $\M_\hbar^\wedge/\hbar \M_\hbar^\wedge$ coincides with the completion
$(\M_\hbar/\hbar \M_\hbar)^\wedge_\chi$ of the $\A_\hbar/(\hbar)$-module $\M_\hbar/\hbar \M_\hbar$
at $\chi$.
\end{enumerate}
\end{Prop}
The proof of this proposition involves
a standard machinery of blow-up algebras, compare with \cite{Eisenbud}, Chapter 7.

For an associative algebra $\A$ and a two-sided ideal $\J\subset \A$ one can form the blow-up
algebra $\Bl_\J(\A)=\bigoplus_{i=0}^\infty \J^i$. This algebra is positively graded.  To ensure nice properties of the completion $\varprojlim_i \A/\J^i$ we need to make sure that the blow-up algebra
$\Bl_\J(\A)$ is Noetherian.

\begin{Lem}\label{Lem:1.23}
Let $\A$ be a $\K[\hbar]$-algebra and $\J$ be a two-sided ideal of $\A$ containing $\hbar$.
Suppose that $\A$ is complete and separated in the $\hbar$-adic topology. Further, suppose that
the algebra $\A/(\hbar)$ is commutative  and  Noetherian. Finally, suppose that
$[\J,\J]\subset \hbar \J$. Then the algebra $\Bl_\J(\A)$ is Noetherian.
\end{Lem}
\begin{proof}
Consider the {\it completed blow-up algebra} $\hat{Bl}_\J(\A):=\prod_{i=0}^\infty
\J^i$. Let $\hbar,\hbar'$ denote the images of $\hbar\in \A$ under the embeddings $\A,\J\hookrightarrow \Bl_\J(\A)$. Since $\A$ is complete in the $\hbar$-adic topology,
we see that  $\hat{\Bl}_\J(\A)$ is complete in the $(\hbar,\hbar')$-adic topology.

Consider the algebra $\hat{\Bl}_\J(\A)/(\hbar,\hbar')=\prod_{i=0}^\infty \J^i/\hbar \J^{i-1}$ (here we assume
that $\J^{-1}=\J^0=\A$). It has a decreasing filtration $\F_i\hat{\Bl}_{\J}(\A)/(\hbar,\hbar'):=\prod_{j\geqslant i}\J^j/\hbar \J^{j-1}$. The associated graded algebra is nothing else but $\bigoplus_{i=0}\J^i/\hbar \J^{i-1}=\Bl_\J(\A)/(\hbar,\hbar')$.
Let us show that the last algebra is commutative and Noetherian.

Commutativity of $\Bl_\J(\A)/(\hbar,\hbar')$ means $[\J^i,\J^j]\subset \hbar \J^{i+j-1}$. This follows
easily from $[\J,\J]\subset \hbar \J$. The algebra $\Bl_\J(\A)/(\hbar,\hbar')$ is commutative and generated
by $\J/(\hbar)$ over the Noetherian algebra $\A_\hbar/(\hbar)$. It follows that $\Bl_\J(\A)/(\hbar,\hbar')$ is Noetherian.

Since $\Bl_\J(\A)/(\hbar,\hbar')$ is the associated graded of $\hat{\Bl}_\J(\A)/(\hbar,\hbar')$
with respect to the  complete separated filtration, we see that $\hat{\Bl}_\J(\A)/(\hbar,\hbar')$ is Noetherian.
Since $\hat{\Bl}_\J(\A)$ is complete in the $(\hbar,\hbar')$-adic topology, and the quotient
$\hat{\Bl}_\J(\A)/(\hbar,\hbar')$ is Noetherian, we see that $\hat{\Bl}_\J(\A)$ itself is Noetherian.

Let us show that the Noetherian property for $\hat{\Bl}_\J(\A)$ implies that for $\Bl_\J(\A)$.
More generally, let $B:=\bigoplus_{i\geqslant 0}B_i$ be a $\Z_{\geqslant 0}$-graded algebra
and let $\hat{B}$ be the completion with respect to this grading (e.g., $B=\Bl_\J(\A),
\hat{B}=\hat{\Bl}_\J(\A)$). Suppose that $\hat{B}$ is Noetherian. We want to prove that
$B$ is also Noetherian.

Consider the algebra $\hat{B}[\hbar]$. It is Noetherian. We have a $\K^\times$-action on
$\hat{B}[\hbar]$ defined  as follows: the action on $\hat{B}$ is induced from the grading
on $B$, while $t.\hbar:=t^{-1}\hbar$ for any $t\in \K^\times$. Consider the embedding $B\rightarrow
\hat{B}[\hbar]$ sending $b\in B_i$ to $b\hbar^i$. This embedding gives an identification
$B\cong \hat{B}[\hbar]^{\K^\times}$.

Pick a left ideal $I\subset B=\hat{B}[\hbar]^{\K^\times}$. The left $\hat{B}[\hbar]$-ideal $\hat{B}[\hbar]I$
is generated by elements $j_1,\ldots,j_k\in I$. We claim that $j_1,\ldots, j_k$ generate the left ideal $I\subset B$.
Indeed, let $j\in I$ and let $b_1,\ldots,b_k\in \hat{B}[\hbar]$ be such that $j=\sum_{i=1}^k b_ij_i$.
Write $b_i:=\sum_{l=0}^{d_i} b_{il}\hbar^l$ and $b_{il}:=\sum_{q\geqslant 0}b_{il}^q$ with $ b_{il}^q\in B_q$.
Then $b_i':=\sum_{l=0}^{d_i}b_{il}^l\hbar^l$ lies in $\hat{B}[\hbar]^{\K^\times}$ and
$j=\sum_{i=1}^k b_i' j_i$.
\end{proof}

\begin{proof}[Proof of Proposition \ref{Lem:3.0.1}]
Let us prove the first claim.

First of all, we remark that the algebra $\A_\hbar$ is  Noetherian (this can be proved using the standard  argument of Hilbert, since $\A_\hbar=A[\hbar]$ as a vector space, $\A_\hbar/(\hbar)=A$, and $A$ is Noetherian).

Consider the completion $\A_\hbar'$ of $\A_\hbar$ in the $\hbar$-adic topology.
To any (left) finitely generated $\A_\hbar$-module $\M_\hbar$ one can assign its completion
$\M'_\hbar:=\varprojlim \M_\hbar/ \hbar^k\M_\hbar$,
which has a natural structure of a  $\A'_\hbar$-module. The blow-up algebra
$\Bl_{(\hbar)}(\A_\hbar)$ is nothing else but the polynomial algebra $\A_\hbar[\hbar]$
and, in particular, is Noetherian. So applying the argument of \cite{Eisenbud}, Chapter 7, we see that
\begin{enumerate}
\item the functor of the $\hbar$-adic
completion is exact.
\item $\M_\hbar'=\A'_\hbar\otimes_{\A_\hbar}\M_\hbar$.
\end{enumerate}

Let $I'_{\chi,\hbar}$ be the completion of $I_{\chi,\hbar}$ in the $\hbar$-adic topology.
Lemma \ref{Lem:1.23} applies to $\A:=\A'_\hbar$ and $\J:=I'_{\chi,\hbar}$ because
$[\J,\J]\subset \hbar^2 \A$. So $\Bl_\J(\A)$ is Noetherian. It follows that the Artin-Rees lemma (see, for example,
\cite{Eisenbud}, Chapter 5) holds for $\J\subset \A$.

Following the proof in the commutative case that can be
found in \cite{Eisenbud}, chapter 7, we prove assertion (1).
 Assertions  (3) and (4) follow from the exactness of
$0\rightarrow\M^\wedge_\hbar\xrightarrow{\hbar} \M^\wedge_\hbar\rightarrow (\M_\hbar/\hbar \M_\hbar)^\wedge_\chi\rightarrow 0$,
which stems from (1). Assertion (2) follows from (4).
\end{proof}

In the sequel we will need the following corollary of Proposition \ref{Lem:3.0.1}.

\begin{Cor}\label{Cor:3.0.1}
Let $\I_\hbar$ be a right ideal in $\A_\hbar$, $\M_\hbar$ be an $\A_\hbar$-bimodule
that is finitely generated as a left $\A_\hbar$-module.
Let $\underline{\M}_\hbar$ be the  annihilator (of the right action) of $\I_\hbar$ in $\M_\hbar$. Then the  annihilator of $\I_\hbar$ in $\M^\wedge_\hbar$ coincides with $\underline{\M}^\wedge_\hbar$.
\end{Cor}
\begin{proof}
At first, consider the case where $\I_\hbar$ is generated by a single element, say $a$. Consider the exact sequence
$0\rightarrow \underline{\M}_\hbar\rightarrow \M_\hbar\xrightarrow{\cdot a} \M_\hbar$. By assertion (1)
of Proposition \ref{Lem:3.0.1}, the completion functor is exact. The same assertion implies
 that the completion functor is $\A^{op}_\hbar$-linear. So we get the exact sequence $0\rightarrow \underline{\M}^\wedge_\hbar\rightarrow \M_\hbar^\wedge\xrightarrow{\cdot a} \M^\wedge_\hbar$. This completes the proof in the case when
$\I_\hbar$ is generated by one element.

Let us proceed to the general case.
The ideal $\I_\hbar$ is generated by some elements $a_1,\ldots,a_k$. Let $\underline{\M}^i_\hbar$ stand for the right annihilator of $a_i$ in $\M_\hbar$. Then $\underline{\M}_\hbar=\bigcap_i \underline{\M}^i_\hbar$.
Since the completion functor is exact, we see that $\underline{\M}^\wedge_\hbar=\bigcap_i \underline{\M}^{i\wedge}_\hbar$. To complete the proof apply the result of the previous
paragraph.
\end{proof}

\begin{Lem}\label{Lem:1.231}
\begin{enumerate}
\item The algebra $\A^\wedge_\hbar$ is Noetherian.
\item Any finitely generated left $\A^\wedge_\hbar$-module is complete and separated with respect to
$I^\wedge_{\chi,\hbar}$-adic topology.
\item Any submodule in a finitely generated $\A^\wedge_\hbar$-module is closed with respect
to $I^\wedge_{\chi,\hbar}$-adic topology.
\end{enumerate}
\end{Lem}
\begin{proof}
Assertion (1) is easy, for example, it follows from the observation that $\A^\wedge_\hbar$ is complete in the $\hbar$-adic topology, and $\A^\wedge_\hbar/(\hbar)=A^\wedge_\chi$ is Noetherian. To prove (2) and (3) we notice that
Lemma \ref{Lem:1.23} applies to $\A:=\A^\wedge_\hbar, \J:=I^\wedge_{\chi,\hbar}$. So the Artin-Rees lemma holds
for $\J\subset \A$. (3) and the claim in (2) that the filtration  is complete  are direct corollaries of the Artin-Rees lemma. The claim in (2) that the filtration is separated  is proved in the same way as the Krull separation theorem,
 compare with \cite{Eisenbud}, Section 5.3.
\end{proof}

\subsection{Harish-Chandra bimodules over quantum algebras}\label{SUBSECTION_completions}
In this subsection we will introduce categories of Harish-Chandra bimodules
for the quantum algebras $\U_\hbar,\Walg_\hbar,\U^\wedge_\hbar,\Walg^\wedge_\hbar$
and their $Q$-equivariant analogs.

Equip $\U_\hbar$ with the "doubled" usual $\K^\times$-action ($t.\xi=t^2\xi, t.\hbar=t\hbar, t\in \K^\times,
\xi\in\g$) and $\Walg_\hbar$ with the Kazhdan $\K^\times$-action. For $\A_\hbar=\U_\hbar$ or $\Walg_\hbar$ we say that
a graded  $\A_\hbar$-bimodule $\M_\hbar$, where the left and the right actions of $\K[\hbar]$ coincide, is Harish-Chandra if
\begin{itemize}
\item[(i)] $\M_\hbar$ is $\K[\hbar]$-flat.
\item[(ii)] $\M_\hbar$ is finitely generated as a $\A_\hbar$-bimodule.
\item[(iii)] $[a,m]\in \hbar^2 \M_\hbar$ for any $a\in \A_\hbar, m\in \M_\hbar$.
\end{itemize}

The following lemma describes simplest properties of Harish-Chandra bimodules.
\begin{Lem}\label{Lem:2.5.1}
\begin{enumerate}
\item $\M_\hbar$ is finitely generated both as a left and as a right $\A_\hbar$-module.
\item All graded components of $\M_\hbar$ are finite dimensional.
\item For $\A_\hbar=\U_\hbar$ the adjoint action of $\g$ on $\M_\hbar$: $(\xi,m)\mapsto \frac{1}{\hbar^2}[\xi,m],\xi\in\g,m\in \M_\hbar,$
is locally finite.
\end{enumerate}
\end{Lem}
\begin{proof}
 Let $m_1,\ldots,m_k$ be homogeneous generators of the  $\A_\hbar$-bimodule
$\M_\hbar$ and let $\underline{\M}_\hbar$ be the left submodule in $\M_\hbar$ generated by
$m_1,\ldots,m_k$. From  (iii) it follows that $\M_\hbar=\underline{\M}_\hbar+\hbar \M_\hbar$.
But $\A_\hbar$ is positively graded. So (ii) implies that the grading on $\M_\hbar$
is bounded from below. Now the proof of (1) follows easily.

(2) follows from (ii).

To prove (3) we note  that the map
$\M_\hbar\rightarrow \M_\hbar, m\mapsto \frac{1}{\hbar^2}[\xi,m],\xi\in\g,$ preserves the grading.
\end{proof}

Let $\M\in \HC(\U)$. Slightly modifying a standard definition, we say that a filtration
$\F_i\M$ is good if it is $\ad(\g)$-stable,  compatible with the "doubled" standard filtration
$\F_i\U:=\F_{[i/2]}^{st}\U$ on $\U$, and $\gr\M$ is a finitely generated $\gr\U=S(\g)$-module.
To construct a good filtration take an $\ad(\g)$-stable finite dimensional subspace $M^0\subset \M$
generating $\M$ as a bimodule, and set $\F_i\M:=\F_i\U M_0$.

Given a good filtration $\F_i\M$ on $\HC(\U)$ form the Rees $\U_\hbar=R_\hbar(\U)$-bimodule $R_\hbar(\M)=
\bigoplus_{i\in \Z}\hbar^i \F_i\M$. Then $R_\hbar(\M)$  becomes an object in $\HC(\U_\hbar)$,
the $i$-th graded component being $\hbar^i \F_i\M$. Conversely, for $\M_\hbar\in \HC(\U_\hbar)$
the quotient $\M_\hbar/(\hbar-1)\M_\hbar$ lies in $\HC(\U)$ and the filtration induced by the grading
on $\M_\hbar$ is good. It is clear that the assignments $\M\mapsto R_\hbar(\M), \M_\hbar\mapsto
\M_\hbar/(\hbar-1)\M_\hbar$ are quasiinverse functors between the category of Harish-Chandra
bimodules equipped with a good filtration and the category $\HC(\U_\hbar)$.

Now let $\N$ be a $\Walg$-bimodule. We say that $\N$ is {\it Harish-Chandra} if there is
a filtration $\F_i\N$ on $\N$ such that $R_\hbar(\N)\in \HC(\Walg_\hbar)$, equivalently,
$\gr\N$ is a finitely generated $\K[S]$-module, and the filtration $\F_i\N$ is almost
commutative in the sense that $[\KF_i\Walg,\F_j\N]\subset \F_{i+j-2}\N$.

Recall the subgroup $Q:=Z_G(e,h,f)\subset G$.
Now let us define $Q$-equivariant Harish-Chandra $\Walg_\hbar$-bimodules. We say that a HC $\Walg_\hbar$-bimodule
$\N_\hbar$ is $Q$-equivariant if it is equipped with a $Q$-action such that
\begin{itemize}
\item[(i$Q$)] The $Q$-action preserves the grading.
\item[(ii$Q$)] The structure map $\Walg_\hbar\otimes \N_\hbar\otimes \Walg_\hbar\rightarrow \N_\hbar$ is $Q$-equivariant.
\item[(iii$Q$)] The differential of the $Q$-action on $\N_\hbar$ coincides with the the action
of $\q\hookrightarrow \Walg_\hbar$ given by $\frac{1}{\hbar^2}[\xi,\cdot], \xi\in\q$.
\end{itemize}

Analogously we define the category $\HC^Q(\Walg)$ of $Q$-equivariant Harish-Chandra $\Walg$-bimodules.

Now let us introduce suitable categories for the completed algebras $\U^\wedge_\hbar, \Walg^\wedge_\hbar$.
Consider the Kazhdan actions of $\K^\times$ on $\A_\hbar^\wedge:=\U^\wedge_\hbar$ or $\Walg^\wedge_\hbar$.

We say that a $\K^\times$-weakly equivariant  $\A^\wedge_\hbar$-bimodule $\M'_\hbar$,
where the left and the right actions of $\K[[\hbar]]$ coincide, is Harish-Chandra if
\begin{itemize}
\item[(i$^\wedge$)] $\M'_\hbar$ is $\K[[\hbar]]$-flat.
\item[(ii$^\wedge$)] $\M'_\hbar$ is a  finitely generated $\A^\wedge_\hbar$-bimodule
and is complete in the $I^\wedge_{\chi,\hbar}$-adic topology.
\item[(iii$^\wedge$)] $[a,m]\in \hbar^2\M'_\hbar$ for any $a\in \A^\wedge_\hbar, m\in \M'_\hbar$.
\end{itemize}

We remark that (ii$^\wedge$) and (iii$^\wedge$) easily imply that $\M_\hbar'$
is finitely generated both as a left and as a right $\A^\wedge_\hbar$-module. Conversely, any finitely generated
left $\A^\wedge_\hbar$-module is complete in the $I^\wedge_{\chi,\hbar}$-adic topology, see Lemma \ref{Lem:1.231}.
The category of Harish-Chandra $\A^\wedge_\hbar$-bimodules will be denoted by $\HC(\A^\wedge_\hbar)$.

The definition of a $Q$-equivariant Harish-Chandra $\A^\wedge_\hbar$-bimodule is given by analogy with
that of a $Q$-equivariant Harish-Chandra $\Walg_\hbar$-bimodule (one should replace (i$Q$) with the condition
that the $Q$-action commutes with the $\K^\times$-action). The category of $Q$-equivariant
Harish-Chandra $\A^\wedge_\hbar$-bimodules  is denoted by $\HC^Q(\A^\wedge_\hbar)$.

We remark that the categories $\HC(\U_\hbar),\HC^Q(\U^\wedge_\hbar)$ etc. are $\K[\hbar]$-linear  but
not abelian (due to the flatness condition, cokernels are undefined in general). It still makes
sense to speak about exact sequences in our categories. Also the categories under  consideration
have tensor product functors: for instance, for $\M^1_\hbar,\M^2_\hbar\in \HC(\A^\wedge_\hbar)$,
one can take the usual tensor product $\M_\hbar\otimes_{\A^\wedge_\hbar}\N_\hbar$ of bimodules and then take its quotient by the $\hbar$-torsion. So this tensor product satisfies (i$^\wedge$). Clearly, it satisfies (iii$^\wedge$). To see that it satisfies (ii$^\wedge$) we remark that $\M^1_\hbar\otimes_{\A^\wedge_\hbar}\M^2_\hbar$
is finitely generated as a left $\A^\wedge_\hbar$-module, because $\M^1_\hbar,\M^2_\hbar$ are.

The categories $\HC(\U_\hbar)$ and $\HC^Q(\U^\wedge_\hbar)$ are related via the completion functor
as explained in the following lemma.

\begin{Lem}\label{Lem:2.5.2}
Let $\M_\hbar\in \HC(\U_\hbar)$.
\begin{enumerate}
\item The completion $\M^\wedge_\hbar:=\varprojlim \M_\hbar/I_{\chi,\hbar}^k \M_\hbar$
has a natural structure of a $Q$-equivariant Harish-Chandra $\U^\wedge_\hbar$-bimodule.
\item The completion functor $\HC(\U_\hbar)\rightarrow \HC^Q(\U^\wedge_\hbar)$ is exact and tensor.
\end{enumerate}
\end{Lem}
\begin{proof}
To see that $\M^\wedge_\hbar$ is indeed a $\U^\wedge_\hbar$-bimodule we remark that $I_{\chi,\hbar}^k\M_\hbar=\M_\hbar
I_{\chi,\hbar}^k$ thanks to (iii).  Equip $\M_\hbar$ with a Kazhdan $\K^\times$-action: $(t,m)\mapsto t^{2i}\gamma(t)m$ for $m$ of degree $i$ and with a $Q$-action restricted from the $G$-action (the latter is integrated from the adjoint $\g$-action, $\xi\mapsto \frac{1}{\hbar^2}\ad(\xi)$). It is straightforward to verify that
$\M^\wedge_\hbar$ becomes an object of $\HC^Q(\U^\wedge_\hbar)$.

(2) follows from assertion (1) of Proposition \ref{Lem:3.0.1}.
\end{proof}

Similarly, we have a completion functor $\HC^Q(\Walg_\hbar)\rightarrow \HC^Q(\Walg^\wedge_\hbar)$.

\section{Construction of functors}\label{SECTION_functors}
\subsection{Correspondence between
ideals}\label{SUBSECTION_corr_ideals} Here we recall the
construction of mappings between the sets $\Id(\Walg),\Id(\U)$, see  \cite{Wquant}, Subsection 3.4
for details and proofs.

Recall the algebras $\U_\hbar,\U_\hbar^\wedge, \Walg_\hbar,\Walg_{\hbar}^\wedge, \W_\hbar, \W^\wedge_\hbar,
\W^\wedge_\hbar(\Walg^\wedge_\hbar)$ and the isomorphism $\Phi_\hbar:\U_\hbar^\wedge\rightarrow \W_\hbar^\wedge(\Walg^\wedge_\hbar)$ established in Subsection \ref{SUBSECTION_decomp}.

The map $\I\mapsto \I^\dagger: \Id(\Walg)\rightarrow
\Id(\U)$ is constructed as follows. Equip $\I$ with the filtration restricted from the Kazhdan filtration on $\Walg$. Construct the
ideal $\I_\hbar:=R_\hbar(\I)\subset R_\hbar(\Walg)=\Walg_\hbar$ and
take its completion $\I^\wedge_\hbar\subset \Walg^\wedge_\hbar$. Then
construct the ideal
$\W^\wedge_\hbar(\I^\wedge_\hbar):=\W^\wedge_\hbar\widehat{\otimes}_{\K[[\hbar]]}\I^\wedge_\hbar$
in $\W^\wedge_\hbar(\Walg^\wedge_\hbar)= \U^\wedge_\hbar$. Taking
its intersection with $\U_\hbar\subset \U^\wedge_\hbar$, we get an
ideal $\I^\dagger_\hbar\subset \U_\hbar$. Finally, set
$\I^\dagger:=\I^\dagger_\hbar/(\hbar-1)\I^\dagger_\hbar$. We remark that, by construction,
the map $\I\mapsto \I^\dagger$ is $Q$-invariant: $(q\I)^\dagger=\I^\dagger$ for any $q\in Q$.

To construct a map $\J\mapsto \J_\dagger:
\Id(\U)\rightarrow \Id(\Walg)$ we, first,
pass from $\J$ to $\J_\hbar:=R_\hbar(\J)\subset \U_\hbar$ and then to its completion
$\J^\wedge_\hbar\subset
\U_\hbar^\wedge=\W^\wedge_\hbar(\Walg^\wedge_\hbar)$. The completion is $\hbar$-saturated and hence
has the form $\W^\wedge_\hbar(\I^\wedge_\hbar)$ for a unique
$\K^\times$-stable ideal $\I^\wedge_\hbar$. Then take the
intersection $\I_\hbar:=\I^\wedge_\hbar\cap \Walg_\hbar$ ($\I_\hbar$
is indeed dense in $\I^\wedge_\hbar$) and,
finally, set $\J_\dagger:=\I_\hbar/(\hbar-1)$. The ideal $\J_\dagger$
is $Q$-stable for any $\J$.

These two maps enjoy the following properties (\cite{Wquant},
Theorem 1.2.2 and its proof in Subsection 3.4).

\begin{Thm}\label{Thm:5}
\begin{itemize}
\item[(i)] $(\I_1\cap \I_2)^\dagger=\I_1^\dagger\cap \I_2^\dagger$.
\item[(ii)] $\I\supset (\I^\dagger)_\dagger$ and $\J\subset
(\J_\dagger)^\dagger$ for any $\I\in\Id(\Walg), \J\in\Id(\U)$.
\item[(iii)] $\I^\dagger\cap \Centr(\g)=\I\cap\Centr(\g)$. In the
r.h.s. $\Centr(\g)$ is embedded into $\Walg$ as explained in Subsection \ref{SUBSECTION_Walg}.
\item[(iv)] $\I^\dagger$ is primitive provided   $\I$ is.
\item[(v)] For any primitive $\J\in \Id_\O(\U)$  we have $\{\I\in \Id_{fin}(\Walg)| \I^\dagger=\J\}=\{\I\in \Id_{fin}(\Walg)| \J_\dagger\subset \I\}$.
\item[(vi)] $\codim_{\Walg}\J_\dagger=\mult_{\overline{\Orb}}\U/\J$ provided
$\overline{\Orb}$ is an irreducible component of $\VA(\U/\J)$.
\item[(vii)] Let $\I\in \Id_{fin}(\Walg)$ be  primitive. Then $\Goldie(\U/\I^\dagger)\leqslant
\Goldie(\Walg/\I)= (\dim\Walg/\I)^{1/2}$. Here $\Goldie$ stands for
the Goldie rank.
\end{itemize}
\end{Thm}

\begin{Rem}\label{Rem:1.1.10}
Actually, $\J_\dagger\cap \Centr(\g)\supset\J\cap \Centr(\g)$. This can be proved either using
an alternative description of $\J_{\dagger}$ given in Subsection \ref{SUBSECTION_Ginzburg} or deduced directly
from the construction of $\J_{\dagger}$ above in this subsection.
 \end{Rem}

\begin{Rem}\label{Rem:1.1.11}
Actually, one can show that $\I\in \Id_{fin}(\Walg)$ implies $\I^\dagger\in \Id_{\O}(\U)$.
Indeed, we have seen in \cite{Wquant} that this holds provided $\I$ is primitive. The proof is based on  the Joseph irreducibility theorem: $\VA(\U/\J)$ is irreducible for any primitive
ideal $\J\subset \U$. Let us deduce the assertion for an arbitrary $\I\in \Id_{fin}(\Walg)$. Assertion (i)
of Theorem \ref{Thm:5} shows that it holds for all semiprime ideals from $\Id_{fin}(\Walg)$.
Tracking the construction of $\I\mapsto \I^{\dagger}$ one can see that $(\I^\dagger)^k\subset (\I^k)^\dagger$.
This yields the assertion in the general case.

However, it is possible to prove that $\I^\dagger\in \Id_{\O}(\U)$ for $\I\in \Id_{fin}(\Walg)$ without
referring to the Joseph theorem and deduce the latter from here.  We will
make a remark about this, Remark \ref{Rem:4.5.1}. The only part of Theorem \ref{Thm:5} used in
Remark \ref{Rem:4.5.1} is (ii), which is pretty straightforward from the constructions.
\end{Rem}

\subsection{Homogeneous vector bundles}\label{SUBSECTION_classical}
In this subsection we will establish category equivalences between various ramifications of
the category of homogeneous vector bundles. As we pointed out in Subsection \ref{SSS_3}, the results of
this section should be considered as induction steps for the proofs of results on Harish-Chandra
$\U_\hbar$ and $\U^\wedge_\hbar$-bimodules.

Let $G$ be an arbitrary connected reductive algebraic group and $H$ be a subgroup
of $G$ such that
\begin{itemize}
\item[(A)] $G/H$ is quasi-affine ($H$ is {\it observable} in the terminology of \cite{Grosshans}).
\item[(B)] $\K[G/H]$ is finitely generated.
\end{itemize}

Consider the category $\HVB_{G/H}$ of homogeneous vector bundles on $G/H$, i.e., of $G$-equivariant
coherent sheaves on $G/H$.

\begin{Lem}\label{Lem:hom1}
\begin{enumerate}
\item If $H\subset G$ satisfies (A),(B), then $H^\circ$ does.
\item If $H\subset G$ satisfies (A),(B), then $\Gamma(G/H,M)$ is a finitely generated $\K[G/H]$-module
for any $M\in \HVB_{G/H}$.
\item For any $x\in \g\cong \g^*$ the stabilizer $G_x$ satisfies (A),(B).
\end{enumerate}
\end{Lem}
\begin{proof}
(1): For $H^\circ$ assertion (A) follows from \cite{Grosshans}, Corollary 2.3, and (B)
follows from \cite{Grosshans}, Theorem 4.1.


(2): This follows from \cite{Grosshans}, Lemma 23.1(h).

(3): This follows \cite{Grosshans}, Theorem 4.3, since $\overline{Gx}$ consists of finitely many orbits
of even dimension.
\end{proof}

We need  the following categories related to $\HVB_{G/H}$. First of all, we consider
the category $\Mod_H$ of finite dimensional $H$-modules.

To construct the second category fix an affine $G$-variety $X$ such that there is an open $G$-equivariant embedding
$G/H\hookrightarrow X$ with $\codim_{X}X\setminus (G/H)\geqslant 2$. We remark that $\K[G/H]$ is the normalization
of $\K[X]$. Then let $\Coh^G(X)$ denote the
category of $G$-equivariant coherent sheaves on $X$. This category has the Serre subcategory
consisting of all modules supported on $X\setminus G/H$. The quotient category will be denoted
by $\HVB^X_{G/H}$. 

Finally, fix an $H$-stable point $x\in G/H$.  Consider the category $\HVB^\wedge_{G/H}$ consisting of all finitely generated $\K[G/H]^\wedge_x$-modules $M$ equipped additionally with actions of $\g$ and $H$ subject to the following compatibility conditions:
\begin{itemize}
\item[(a)] The action map $\K[G/H]^\wedge_x\otimes M\rightarrow M$ is $\g$- and $H$-equivariant.
\item[(b)] The action map $\g\otimes M\rightarrow M$ is $H$-equivariant.
\item[(c)] The differential of the $H$-action on $M$ coincides with the restriction of
the $\g$-action to $\h$.
\end{itemize}

We have various functors between the categories in interest. For instance, to $P\in \Mod_H$ we can assign
the homogeneous vector bundle $\Fun_1(P):=G*_HP$ on $G/H$ with fiber $P$, see, for instance, \cite{VP}, Section 4.8, for details. It is clear that the functor $\Fun_1$ is an equivalence, a quasi-inverse equivalence
$\Fun_1^{-1}$ is given by taking the fiber at $x$. Further, we have the completion functor $\Fun_2:\HVB_{G/H}\rightarrow\HVB^\wedge_{G/H}$. The $H$-action on the completion comes from the
$H$-action on $G*_HP$ restricted from the $G$-action. It is given by the equality $h.(g*_{H}v)=hgh^{-1}*_{H}hv, g\in G,h\in H, v\in P$.

Then we can consider the functor $\Fun_3:\HVB^\wedge_{G/H}\rightarrow
\Mod_H, M\mapsto M/\m_x M$, where
 $\m_x$ denotes the maximal ideal in $\K[G/H]^\wedge_x$. It is clear that $\Fun_3\circ\Fun_2\circ\Fun_1=\operatorname{id}$.
Finally, we have the functor $\widetilde{\Fun}_4: \Coh^G(X)\rightarrow \HVB_{G/H}$ of restriction to
$G/H$. It induces the functor $\Fun_4:\HVB^X_{G/H}\rightarrow \HVB_{G/H}$.

\begin{Prop}\label{Prop:hom2}
The functors $\Fun_2,\Fun_3,\Fun_4$ are equivalences.
\end{Prop}
\begin{proof}
Let us show that $\Fun_4$ is an equivalence. Thanks to assertion (2) of Lemma \ref{Lem:hom1},
 $\Gamma(G/H,M)$ is  finitely generated as a $\K[G/H]$- and hence also as a $\K[X]$-module for any $M\in \HVB^\wedge_{G/H}$. So we can consider
 $\Gamma(G/H,M)$ as an object of $\Coh^G(X)$.  Let $\Fun_4'(M)$ be the image of
 $\Gamma(G/H,M)$ in $\HVB^X_{G/H}$. It is clear that $\Gamma(G/H,\bullet)$ is right adjoint to
 $\widetilde{\Fun}_4$. Since  $M$ can be extended to at least some coherent
   sheaf on $X$ and $X$ is affine, we see that the fiber of $\Gamma(G/H,M)$
 in $x$ is the same as that of $M$. This observation implies that $\Fun_4'$ is a two-sided inverse to
 $\Fun_4$.

Now let us show that $\Fun_3$ is an equivalence. For this it is enough to show that $\Fun_2\circ\Fun_1\circ\Fun_3$
is isomorphic to the identity functor. So we need to produce a functorial isomorphism $M\rightarrow M':=\Fun_2\circ \Fun_1(P)$, where $P:=\Fun_3(M)=M/\m_x M$. To do this we will need to interpret $M'$ differently.
Namely, consider the space $\Hom_\h(\U,P)$ of linear maps $\varphi:\U\rightarrow P$
that are $\h$-equivariant in the sense that $\varphi(\xi u)=\xi.\varphi(u)$ for any $u\in \U, \xi\in\h$.
To establish an isomorphism $M\rightarrow M'$ we will proceed as follows:
\begin{enumerate}
\item construct some natural maps $\iota:M\rightarrow \Hom_\h(\U,P),\iota':M'\rightarrow \Hom_\h(\U,P)$,
\item equip $\Hom_\h(\U,P)$ with $\g$- and $H$-actions and with a $\K[G/H]^\wedge_x$-module
structure so that $\iota,\iota'$ become $\g$- and $H$-equivariant $\K[G/H]^\wedge_x$-module homomorphisms,
\item Show that $\iota,\iota'$ are injective,
\item Show that $\iota'$ is surjective,
\item Show that $\iota'^{-1}\circ \iota$ is surjective.
\end{enumerate}
Since the pair $(M',\iota')$ is a special case of $(M,\iota)$ (because $\Fun_3(M)=P$), it is enough to perform steps (1)-(3) only for $(M,\iota)$.

(1): Let $\pi:M\rightarrow P$ denote the projection.
Recall that we have a $\g$- and hence a $\U$-action on $M$. Map a pair $(u,m),u\in \U, m\in M,$ to
$\pi(um)$. This is an $H$-equivariant bilinear map, in particular,  $\pi(\xi u m)=\xi \pi(um)$ for
any $\xi\in \h$. In other words, we get a linear map $\iota:M\rightarrow \Hom_\h(\U,P)$.

(2): The group $H$ and the algebra $\g$ act on $\Hom_\h(\U,P)$ by $[\eta.\varphi](u)=\varphi(u\eta), [h.\varphi](u)=h\varphi(h^{-1}.u), \eta\in \g, h\in H$ (it is straightforward to check
that the actions are well-defined). The map $\iota$ becomes $\g$- and $H$-equivariant.

Also $\Hom_\h(\U,P)$ is a $\K[G/H]^\wedge_x$-module: for $f\in \K[G/H]^\wedge_x,
\varphi\in \Hom(\U,P)$ we define  $f\varphi$ using the Leibnitz rule: if $[f\varphi](u)$
is already defined, then we set $[f\varphi](\eta u):=(\eta.f)(x)\varphi(u)+ f(x)\varphi(\eta u)$.
It is clear that the multiplication is well-defined and $\iota$  is
$\K[G/H]^\wedge_x$-linear. Finally, it is straight-forward to check that the $\K[G/H]^\wedge_x$-, $\g$-,
and $H$-actions we have introduced satisfy the compatibility conditions (a)-(c) above.

(3): Let us check that
$\iota$ is injective.
Assume the converse, then there is $m\neq 0$ with $\pi(um)=0$ for all $u\in \U$.
But the velocity vector fields $\xi_*, \xi\in\g,$ span the $\K[G/H]^\wedge_x$-module $\Der(\K[G/H]^\wedge_x)$. It follows that
$m\in \bigcap_n \m_x^n M$. Since $M$ is finitely generated, the latter intersection is zero
by the Krull separation theorem.


(4): Let us show that $\iota'$ is surjective. Pick a subspace $V\subset \g$ complimentary to $\h$.
We have a natural map $\tau:\Hom_\h(\U,P)\rightarrow \Hom(S V,P)$ of restriction to $S V\subset \U$ and this map
is injective thanks to the PBW theorem. Let us recall that $M'$ is nothing but the space of
sections of $\Fun_1(P)=G*_HP$ over the formal neighborhood of $x$ that is just a formal
polydisc. The velocity vector map $\xi\mapsto \xi_x$ identifies
$V$ with the tangent space to $G/H$ at $x$. From this it follows that $\tau\circ \iota':M'\rightarrow \Hom_\h(SV,P)$ is surjective. So $\iota'$ is surjective.

(5): Let us check that $\iota'^{-1}\circ \iota$ is surjective. Since $M'$ is finitely generated,
to do this it is enough to prove that the induced map $M/\m_x M\rightarrow M'/\m_x M'$ is surjective.
Analyzing the defintions
of $\iota,\iota'$ we see that the map under consideration is just the identity map $P\rightarrow P$.

Finally, let us remark that, by its construction, the homomorphism $\iota'^{-1}\circ \iota:M\rightarrow M'$ is
functorial.
%
%
\end{proof}

We need an alternative description of the equivalence $\Fun_4\circ \Fun_1\circ \Fun_3:\HVB_{G/H}^\wedge\rightarrow \HVB_{G/H}^X$. This description
will be crucial in the sequel, compare with Subsection \ref{SSS_3}. For a $\g$-module
 $M$ by $M_{G-l.f.}$ we denote the sum of all finite dimensional
submodules in $M_{\g-l.f.}$, where the $\g$-action can be integrated to an algebraic $G$-action.
Of course, if $G$ is simply connected semisimple group (which is our usual convention),
then $M_{G-l.f.}=M_{\g-l.f.}$.

Pick  $M\in \HVB^\wedge_{G/H}$. Consider the subspace $M_{G-l.f.}\subset M$.
It is $H$-stable, let $\rho$ denote the corresponding representation
of $H$ in $M_{G-l.f.}$. On the other hand, $M$ has the structure of a $G$-module integrated
from the $\g$-action. Restricting the representation of $G$ to $H$, we get the representation $\rho'$
of $H$ in $M_{G-l.f.}$. The action map $\g\otimes M_{G-l.f.}\rightarrow M$ is equivariant with respect
to both representations $\rho,\rho'$.
So $\rho(h)\rho'(h)^{-1}$ commutes with $G$ for any $h\in H$. The differentials of $\rho,\rho'$ coincide hence $\rho(h)=\rho'(h)$ for any $h\in H^\circ$.  So
$\sigma(h)=\rho(h)\rho'(h)^{-1}$ defines a representation of
$H/H^\circ$ in $M_{G-l.f.}$ commuting with $G$. The following proposition shows that
$\Fun_4\circ \Fun_1\circ \Fun_3(M)$ coincides with the image of $M_{G-l.f.}^{H/H^\circ}$
in $\HVB^X_{G/H}$.

\begin{Prop}\label{Prop:hom3}
Let $P\in \Mod_H$ and $M=\Fun_2\circ \Fun_1(P)$. Then
\begin{enumerate}\item $M_{G-l.f.}=\Gamma(G/H^\circ, G*_{H^\circ}P)$,
\item $M_{G-l.f.}^{H/H^\circ}=\Gamma(G/H,G*_HP)(=\widetilde{\Fun}_4\circ \Fun_1(P))$.
\end{enumerate}
\end{Prop}
\begin{proof}
(1): Fix a point $\widetilde{x}\in G/H^\circ$ mapping to $x$ under the natural projection
$G/H^\circ\rightarrow G/H$.

The objects in $\HVB^\wedge_{G/H^\circ}$ and $\HVB^\wedge_{G/H}$ obtained from $P\in \Mod_{H}$ are naturally
isomorphic. So we have a natural map $N:=\Gamma(G/H^\circ, G*_{H^\circ}P)\rightarrow M$. This map is injective:
any section in its kernel vanishes in $x$ together with its $G$-translates
and hence vanishes everywhere. Since $N=(\K[G]\otimes P)^H\subset \K[G]\otimes P$ and $\K[G]\otimes P$
is clearly $G$-l.f.,
we see that $N\subset M_{G-l.f.}$.

First of all, let us check that $N^G=(M_{G-l.f.})^G=M^\g$. Recall that we have the projection
$\pi:M\rightarrow P$. Analogously to the proof of the injectivity of $\iota$ in
Proposition \ref{Prop:hom2} one gets that
the restriction of $\pi$ to $M^\g$ is injective. Also it is clear that $\pi(M^\g)\subset P^\h$.
On the other hand, the restriction of $\pi$ to $N$ is nothing else but taking the value of
a section at $\widetilde{x}$. So $\pi(N^G)=P^{H^\circ}=P^\h$. It follows that $N^G=(M_{\g-l.f.})^G$.

To prove that $N=M_{G-l.f.}$ one needs to verify that the spaces of $G$-equivariant linear maps $\Hom_G(L,N)\subset\Hom_G(L,M_{\g-l.f.})$ coincide
for an arbitrary irreducible $G$-module $L$. But $\Hom_G(L,N)=(L^*\otimes N)^G=\Gamma(G/H^\circ,
G*_{H^\circ}( L^*\otimes P))$ and $\Hom_G(L,M_{G-l.f.})=(L^*\otimes M)^\g$. It is clear that $\Fun_2\circ \Fun_1(L^*\otimes P)=L^*\otimes M$.   The equality $\Hom_G(L,N)=\Hom_G(L,M_{G-l.f.})$ follows from the previous
paragraph.

(2): To prove the claim we will need to describe the action $\sigma$ of $H/H^\circ$ on $M_{G-l.f.}=N$. Let us note that $H/H^\circ$ acts on the total space of $G*_{H^\circ}P$ by $h.(g*_{H^\circ}v)=gh^{-1}*_{H^\circ}hv, h\in H$.
Let us prove that $\sigma$ is induced by this action.

The  $G$-action on $M$ is induced by the action $g_1.(g*_{H^\circ}v)=g_1g*_{H^\circ}v$. So $\rho'$ is induced from the action $h.(g*_{H^\circ}v)=hg*_{H^\circ}v$. On the other hand, the description of the $H$-action  on $M=\Fun_2\circ \Fun_1(P)$ implies that $\rho$ is induced from the action $h.(g*_{H^\circ}v)=hgh^{-1}*_{H^\circ}hv$.
So $\sigma(h)=\rho(h)\rho'(h)^{-1}$ is induced from the required $H/H^\circ$-action.

Now it remains to notice that $N^{H/H^\circ}=\Gamma(G/H, G*_HP)$.
\end{proof}

%


Below we will use results of this subsection in the special case  $H:=G_\chi, X:=\overline{\Orb}$. The subgroup $H$ enjoys the properties (A),(B) by assertion (3) of Lemma \ref{Lem:hom1}.

\subsection{Construction of functors between $\HC_{\overline{\Orb}}(\U_\hbar), \HC^Q_{fin}(\Walg_\hbar)$}\label{SUBSECTION_Construction}
In this subsection we  construct  functors
$\bullet_\dagger:\HC(\U_\hbar)\rightarrow \HC^Q(\Walg_\hbar),
\bullet^\dagger: \HC_{fin}^Q(\Walg_\hbar)\rightarrow
\HC_{\overline{\Orb}}(\U_\hbar)$, where the last two categories are defined as follows.
$\HC_{\overline{\Orb}}(\U_\hbar)$ is the full subcategory in $\HC(\U_\hbar)$ consisting of
all bimodules $\M_\hbar$ with $\VA(\M_\hbar)\subset \overline{\Orb}$, while $\HC^Q_{fin}(\Walg)$
is the full subcategory in $\HC^Q(\Walg_\hbar)$ consisting of all bimodules of finite
rank over $\K[\hbar]$.

Recall that the algebras $\U^\wedge_\hbar$ and
$\W^\wedge_\hbar(\Walg^\wedge_\hbar)$ are identified by means of the isomorphism
$\Phi_\hbar$ introduced in the end of Subsection \ref{SUBSECTION_decomp}.

The next proposition is essential in our construction.

\begin{Prop}\label{Prop:3.4.1}
The  categories $\HC^Q(\Walg_\hbar),\HC^Q(\Walg^\wedge_\hbar),\HC^Q(\U^\wedge_\hbar)$ are equivalent.
Quasiinverse equivalences look as follows:
\begin{itemize}
\item $\HC^Q(\Walg_\hbar)\rightarrow \HC^Q(\Walg^\wedge_\hbar)$: $\Nil_\hbar\mapsto \Nil^\wedge_\hbar$.
\item $\HC^Q(\Walg^\wedge_\hbar)\rightarrow \HC^Q(\Walg_\hbar)$: $\Nil'_\hbar\mapsto
(\Nil'_\hbar)_{\K^\times-l.f.}$.
\item $\HC^Q(\Walg^\wedge_\hbar)\rightarrow \HC^Q(\U^\wedge_\hbar)$: $\Nil'_\hbar\mapsto \W_\hbar^\wedge(\Nil'_\hbar):=
\W^\wedge_\hbar\widehat{\otimes}_{\K[[\hbar]]}\Nil'_\hbar$.
\item $\HC^Q(\U^\wedge_\hbar)\rightarrow \HC^Q(\Walg^\wedge_\hbar)$: $\M'_\hbar\mapsto (\M'_\hbar)^{\ad V}$.
\end{itemize}
\end{Prop}
\begin{proof}
 To prove that the first two functors are mutually quasiinverse one
needs to check that:
\begin{enumerate}
\item $\Nil_\hbar$ coincides with the $\K^\times$-l.f. part of its
completion.
\item For any $\Nil'_\hbar$ its $\K^\times$-l.f. part is dense in
$\Nil'_\hbar$ and is a finitely generated left $\Walg_\hbar$-module.
\end{enumerate}
Both claims  follow from the fact that $\Walg_\hbar$ is
positively graded. For reader's convenience we give proofs here.

Let us prove (1). For a $\K^\times$-module $N$ let $N^{(i)}$ denote the space of all vectors $v$
with $t.v=t^i v$.   Let $\m$ denote the maximal ideal of 0 in $\Walg_\hbar$.
Then $\m^{(i)}=\{0\}$ for $i\leqslant 0$ and $\Nil_\hbar^{(i)}=\{0\}$ for $i\ll 0$.
Also $\Nil_\hbar/ \m^m \Nil_\hbar$
is a  finite dimensional vector space with an algebraic action of $\K^\times$.
For $m>m_0$ we have the exact sequence $(\m^{m_0} \N_\hbar)^{(i)}\rightarrow (\Nil_\hbar/ \m^m \Nil_\hbar)^{(i)}\rightarrow (\Nil_\hbar/ \m^{m_0}\Nil_\hbar)^{(i)}\rightarrow 0$. It follows that for any $i$ there is $m_0$ such that   $(\Nil_\hbar/ \m^m \Nil_\hbar)^{(i)}\xrightarrow{\sim} (\Nil_\hbar/ \m^{m_0}\Nil_\hbar)^{(i)}$
 for $m>m_0$. In particular, $(\Nil_\hbar^\wedge)^{(i)}$ projects isomorphically onto $(\Nil_\hbar/\m^m \Nil_{\hbar})^{(i)}$ for $m\gg 0$. Hence the natural map $\Nil_\hbar^{(i)}\rightarrow (\Nil_\hbar^\wedge)^{(i)}$
is an isomorphism. This completes the proof of (1).

Let us prove (2). We need to check two claims: first, that $(\Nil'_\hbar)_{\K^\times-l.f.}$ is dense
in $\Nil'_\hbar$ and, second, that $(\Nil'_\hbar)_{\K^\times-l.f.}$ is finitely generated as a $\Walg_\hbar$-module.
The first one is equivalent to the assertion that the natural projection $(\Nil'_\hbar)_{\K^\times-l.f.}\rightarrow \Nil'_\hbar/\Nil'_\hbar \m^n$ is surjective for any $n$ or equivalently that $(\Nil'_\hbar)^{(i)}\twoheadrightarrow
(\Nil'_\hbar/ \Nil'_\hbar\m^n)^{(i)}$. This is checked analogously to  the proof of (1). Proceed to the second claim. Since $(\Nil'_\hbar)_{\K^\times-l.f.}$ is dense in $\Nil'_\hbar$, we see that
$(\Nil'_\hbar)_{\K^\times-l.f.}$ generates  $\Nil'_\hbar$
as a left $\Walg^\wedge_\hbar$-module.
Since $\Walg^\wedge_\hbar$ is a Noetherian algebra, we can choose a finite set of generators
of $\Nil'_\hbar$ inside $(\Nil'_\hbar)_{\K^\times-l.f.}$.
It is easy to see that these elements generate $(\Nil'_\hbar)_{\K^\times-l.f.}$ (compare with the last paragraph
of the proof of Lemma \ref{Lem:1.23}).

Now let us check that the last two functors are quasiinverse
equivalences. First, it is a standard (and pretty straightforward to check) fact that the centralizer of $V$
in $\W_\hbar^\wedge$ is $\K$. So
$\Nil'_\hbar=(\W_\hbar^\wedge(\Nil'_\hbar))^{\ad V}$ for any
$\Nil_\hbar'\in \HC^Q(\Walg_\hbar^\wedge)$. It remains  to verify
that the canonical homomorphism $\W^\wedge_\hbar((\M'_\hbar)^{\ad
V})\rightarrow \M'_\hbar$ is an isomorphism for all $\M'_\hbar\in
\HC^Q(\U^\wedge_\hbar)$. 

Fix a symplectic basis $p_1,q_1,\ldots,p_l,q_l$ in $V$ (with
$\omega(p_i,p_j)=\omega(q_i,q_j)=0, \omega(q_i,p_j)=\delta_{ij}$).

Let $\n$ denote the maximal ideal in $\W^\wedge_\hbar$ (generated by $V$ and $\hbar$).
By Lemma \ref{Lem:1.231}, we have $\n \M'_\hbar\neq
\M'_\hbar$. Choose $m_0\in \M'_\hbar\setminus \n \M'_\hbar$.
We claim that there is $m\in (\M'_\hbar)^{\ad V}$ such that  $m-m_0\in \n\M'_\hbar$.
At first, we show that there is $m_0'$  such that $q_1 m_0'=m_0' q_1$
and $m_0'-m_0\in \n\M'_\hbar$. There is $m^1\in \M'_\hbar$
such that $[q_1,m_0]=\hbar^2 m^1$. Then $[q_1,p_1m^1]=\hbar^2
m_1+p_1[q_1,m^1]$. Set $m_1:=m_0-p_1m^1$. So $[q_1,m_1]=-p_1[q_1,m^1]=\hbar^2 p_1
m^2$, where $m^2$ is a unique element of $\M'_\hbar$ with $\hbar^2 m^2=-[q_1,m^1]$.
Put $m_2=m_0-p_1 m^1-\frac{p_1^2}{2}m^2$, then
$[q_1,m_2]=\hbar^2 p_1^2 m^3$ for some $m^3\in \M'_\hbar$.
Define $m^i, i>3,$ in a similar way. Set $m_0':=m_0-\sum_{i=1}^\infty
\frac{1}{i} p_1^i m^i$. Since $\M'_\hbar$ is complete, $m_0'$ is
well-defined. By the construction of the elements $m^i$ the sequence $[q_1,m_0-\sum_{i=1}^l \frac{1}{i}p_1^i m^i]$
converges to zero, hence  $[q_1,m_0']=0$.

Now we will do a similar procedure with $q_1$ instead of $p_1$ and with $m_0'$ instead of $m_0$.
Set $ m'^1:=\hbar^{-2}[p_1,m_0']$, $m'_1:=m_0'+q_1 m^1$, $m'^2:=\hbar^{-2}[p_1,m'^1]$, etc.
We get the element $m_0''=m_0'+q_1 m'^1+\frac{1}{2} q_1^2
m'^2+\ldots$. By construction, all $m'^i$ commute with $q_1$.  So
$m_0''$ commutes with $p_1$ and $q_1$. Repeating the procedure for $p_2,q_2$ and $m_0''$
we get an element in $\M_\hbar'$ congruent with $m$ modulo $\n \M_\hbar'$ and commuting
with $p_1,p_2,q_1,q_2$. Proceeding in the same way for $i=3,\ldots,l$ we get
a required element $m$.

Consider a natural homomorphism
$\W^\wedge_\hbar\left((\M'_\hbar)^{\ad V}\right)\rightarrow \M'_\hbar$.
From the previous paragraph it follows that the image of this homomorphism
projects surjectively to $\M'_\hbar/ \n \M'_\hbar$. But $\M'_\hbar$ is
finitely generated and one can apply an analog of the Nakayama lemma.
So the   homomorphism in  interest is
surjective.

Analogously to the proof of Lemma 3.4.3 in
\cite{Wquant}, any nonzero $\hbar$-saturated subbimodule in
$\W^\wedge_\hbar\left((\M'_\hbar)^{\ad V}\right)$ has nonzero intersection with the  $\ad
V$-invariants, whence the homomorphism is injective.

Finally, since
$\W^\wedge_\hbar\left((\M'_\hbar)^{\ad V}\right)\cong \M'_\hbar$ and $\M'_{\hbar}$ is a finitely generated (=Noetherian) left $\W^\wedge_\hbar(\Walg^\wedge_\hbar)$-module,
we see that $(\M'_\hbar)^{\ad V}$ is a Noetherian and hence a finitely generated left $\Walg^\wedge_\hbar$-module.
So $(\M'_\hbar)^{\ad V}\in \HC^Q(\Walg^\wedge_\hbar)$.
\end{proof}

By definition,  a functor $\bullet_\dagger: \HC(\U_\hbar)\rightarrow
\HC^Q(\Walg_\hbar)$ is the composition of the completion functor
$\HC(\U_\hbar)\rightarrow \HC^Q(\U^\wedge_\hbar)$, see Lemma \ref{Lem:2.5.2}, and the
equivalence $\HC^Q(\U^\wedge_\hbar)\rightarrow \HC^Q(\Walg_\hbar)$ constructed in Proposition \ref{Prop:3.4.1}.

The following lemma summarizes the properties of the functor $\bullet_\dagger$ we need.

\begin{Lem}\label{Lem:3.5.2}
\begin{enumerate}
\item The functor $\bullet_\dagger$ is exact.
\item The functor $\bullet_\dagger$ is tensor.
\item Let $\M_\hbar\in \HC(\U_\hbar)$. Then
$\M_{\hbar\dagger}/\hbar \M_{\hbar\dagger}$ is naturally identified with the pull-back
of the $\K[\g^*]$-module $\M_\hbar/\hbar \M_\hbar$ to $S$.
\item In particular, $\M_{\hbar\dagger}=0$ if
$\VA(\M_\hbar)\cap \Orb=\varnothing$ and
$\M_{\hbar\dagger}\in \HC_{fin}^Q(\Walg_\hbar)$ provided
$\M_\hbar\in \HC_{\overline{\Orb}}(\U_\hbar)$.
Further, if $\M_\hbar\in \HC_{\overline{\Orb}}(\U_\hbar)$, then $\rank_{\K[\hbar]}\M_{\hbar\dagger}=\mult_{\overline{\Orb}}\M_\hbar$.
\end{enumerate}
\end{Lem}
\begin{proof}
(1) follows from the exactness of the completion functor. (2) follows from the observation that the
completion functor as well as the equivalences in Proposition \ref{Prop:3.4.1} are tensor.

The construction of $\bullet_\dagger$ implies that we have a $\K^\times$-equivariant isomorphism
$(\M_\hbar/\hbar \M_\hbar)^\wedge_\chi\cong \K[V^*]^\wedge_0\widehat{\otimes} (\M_{\hbar\dagger}/\hbar\M_{\hbar\dagger})^\wedge_\chi$ of $\K[\g^*]^\wedge_\chi=\K[V^*]^\wedge_0\widehat{\otimes}
\K[S]^\wedge_\chi$-modules. This gives rise to a $\K^\times$-equivariant isomorphism
between $(\M_{\hbar\dagger}/\hbar\M_{\hbar\dagger})^\wedge_\chi$ and the pull-back of $(\M_\hbar/\hbar \M_\hbar)^\wedge_\chi$ to $S^\wedge_\chi$. Then $\M_{\hbar\dagger}/\hbar \M_{\hbar\dagger}$ (resp., the pull-back of
$\M_\hbar/\hbar \M_\hbar$ to $S$) are just the spaces of $\K^\times$-l.f. vectors in the modules in the previous sentence, compare with the proof of the first two equivalences in Proposition \ref{Prop:3.4.1}.
This completes the proof of (3).

Assertion (4) follows from (3).
%
\end{proof}

Now let us produce a functor in the opposite direction. For $\M'_\hbar\in \HC(\U^\wedge_\hbar)$
the subspace $(\M'_\hbar)_{\g-l.f.}\subset \M'_\hbar$ is $\K^\times$-stable. Integrate the
$\g$-action on $(\M'_\hbar)_{\g-l.f.}$ to a $G$-action. Define a new $\K^\times$-action
on $(\M'_\hbar)_{\g-l.f.}$ by composing the existing one with $\gamma(t)^{-1}$. Since $\M'_\hbar$ is a $\K^\times$-weakly equivariant $\U_\hbar$-bimodule (with respect to the Kazhdan $\K^\times$-action on $\U_\hbar$),
the new $\K^\times$-action commutes with
$\g$. Set $(\M'_\hbar)_{l.f.}:=[(\M'_\hbar)_{\g-l.f.}]_{\K^\times-l.f.}$.

Let $\HC_{\overline{\Orb}}(\U^\wedge_\hbar)$ denote the full subcategory in $\HC(\U^\wedge_\hbar)$
consisting of all bimodules $\M'_\hbar$ such that $\VA(\M'_\hbar)$ is contained in
(the completion of) $\overline{\Orb}$.

\begin{Lem}\label{Lem:3.5.1}
If $\M'_\hbar\in \HC_{\overline{\Orb}}(\U^\wedge_\hbar)$, then
$(\M'_{\hbar})_{l.f.}\in
\HC_{\overline{\Orb}}(\U_\hbar)$.
\end{Lem}
\begin{proof}
Set $\M_\hbar:=(\M'_\hbar)_{l.f.}$. Let us check that $\M_\hbar\in \HC(\U_\hbar)$ which amounts to showing that $\M_\hbar$ is finitely generated as a left $\U_\hbar$-module.

Set $M':=\M'_\hbar/\hbar \M'_\hbar, M:=M'_{\g-l.f.}$.
First, we are going to check that $M$ is a finitely generated $\K[\g^*]$-module.
Since $\M'_\hbar\in \HC_{\overline{\Orb}}(\U^\wedge_\hbar)$, we see that $M'$ is annihilated by some power of
the ideal $I(\Orb)$ of $\overline{\Orb}$. Consider the $I(\Orb)$-adic filtration on $M'$. This filtration is finite.
Its quotients are objects in $\HVB^\wedge_{G/(G_\chi)^\circ}$. Indeed, the only thing that we need
to check is that these quotients come equipped with a $(G_\chi)^\circ$-action. But $(G_\chi)^\circ$ is the semidirect product of its unipotent radical and $Q^\circ$. The action of the former is uniquely recovered from the action
of its Lie algebra.

Let $N$ be one of the quotients. As we checked in Subsection \ref{SUBSECTION_classical}, assertion (2) of Lemma \ref{Lem:hom1} and assertion (1) of Proposition \ref{Prop:hom3}, $N_{\g-l.f.}$ is a finitely generated $\K[G/(G_\chi)^\circ]$-module and hence a finitely  generated $\K[\g^*]$-module because  $\K[G/(G_\chi)^\circ]$ is finite  over $\K[\g^*]$. Since the functor of taking $\g$-l.f. sections (on the category of $\g$-modules) is left-exact and all $N_{\g-l.f.}$ are finitely generated, we see that $M$ is finitely generated.

Now let us show that $\M_\hbar/\hbar \M_\hbar$ is a finitely generated
$\K[\g^*]$-module. To do this
consider the exact sequence $0\rightarrow \M'_\hbar\xrightarrow{\hbar}\M'_\hbar\rightarrow M'\rightarrow 0.$
Again, since the functor of taking ($\g$- and $\K^\times$-) l.f. sections is left exact we have the following exact sequence
$0\rightarrow \M_\hbar\xrightarrow{\hbar} \M_\hbar\rightarrow
M'_{l.f.}.$
It follows that the $\K[\g^*]$-module $\M_\hbar/\hbar \M_\hbar$ embeds into $M$ and hence is finitely generated.

Let us proceed with the proof that $\M_\hbar$ is finitely generated.
Lemma \ref{Lem:1.231} implies
that the $\hbar$-adic
filtration on $\M'_\hbar$ is separated. Hence the  same holds for $\M_\hbar$.
We have the grading $\M_\hbar=\bigoplus_{i\in \Z} \M_\hbar^i$. Recall that the multiplication by $\hbar$ increases
the degree by 1.

Pick   generators $m_1,\ldots,m_k$ of the $\K[\g^*]$-module
$\M_\hbar/\hbar\M_\hbar$ of degrees $i_1,\ldots,i_k$.
Lift them to some homogeneous elements
$\widetilde{m}_1,\ldots,\widetilde{m}_k\in \M_\hbar$. We claim that these
elements generate $\M_\hbar$.

Let $\underline{\M}_\hbar$ denote the left submodule of $\M_\hbar$ generated by $\widetilde{m}_i,i=1,\ldots,k$.
It is easy to show by induction that $\M_\hbar=\underline{\M}_\hbar+\hbar^n \M_\hbar$ for any
positive integer $n$. Since the $\hbar$-adic filtration on $\M_\hbar$ is separated, we get that the grading on $\M_\hbar$ is bounded from below. Now the claim in the previous paragraph follows easily.
%
So $\M_\hbar\in \HC(\U_\hbar)$.

It remains to check that
$\VA(\M_\hbar)=\overline{\Orb}$. Since $\M'_\hbar$ is $\K[[\hbar]]$-flat, we see that $m\in \M'_\hbar$
is l.f. if and only if so is $\hbar m$. Equivalently, $\M_\hbar$ is an
$\hbar$-saturated subspace in $\M'_\hbar$.  Therefore $\M_\hbar/\hbar\M_\hbar\subset \M'_\hbar/\hbar\M'_\hbar$
and, in particular, $\M_\hbar/\hbar \M_\hbar$ is annihilated by some power of $I(\Orb)$.
\end{proof}

Similarly to the previous subsection, one can define an action of $C(e)=Q/Q^\circ$ on
$(\M'_\hbar)_{l.f.}$ for any $\M'_\hbar\in\HC^Q_{\overline{\Orb}}(\U^\wedge_\hbar)$.
So we get a functor $\HC^Q_{\overline{\Orb}}(\U^\wedge_\hbar)\rightarrow \HC_{\overline{\Orb}}(\U_\hbar),
\M'_\hbar\mapsto (\M'_\hbar)_{l.f.}^{C(e)}$. Composing this functor
with the equivalence $\HC_{fin}^Q(\Walg_\hbar)\rightarrow \HC^Q_{\overline{\Orb}}(\U^\wedge_\hbar)$
from Proposition \ref{Prop:3.4.1}, we get the functor $\bullet^\dagger: \HC_{fin}^Q(\Walg_\hbar)\rightarrow
\HC_{\overline{\Orb}}(\U_\hbar)$.

Let us describe a relation between $\bullet_\dagger$ and $\bullet^\dagger$.

\begin{Prop}\label{Prop:3.5.5}
\begin{enumerate}
\item $\bullet^\dagger$ is right adjoint to $\bullet_\dagger: \HC_{\overline{\Orb}}(\U_\hbar)\rightarrow
\HC_{fin}^Q(\Walg_\hbar)$. In particular, the functor $\bullet^\dagger$ is left exact.
\item For any $\M\in \HC_{\overline{\Orb}}(\U_\hbar)$ the kernel and the cokernel of the natural morphism $\M_\hbar\rightarrow (\M_{\hbar\dagger})^\dagger$ lie in $\HC_{\partial\Orb}(\U_\hbar)$.
\end{enumerate}
\end{Prop}
\begin{proof}
(1):  Using  the equivalence of Proposition \ref{Prop:3.4.1} we
identify $\HC_{fin}^Q(\Walg_\hbar)$ with
$\HC^Q_{\overline{\Orb}}(\U^\wedge_\hbar)$.  The functor $\bullet_\dagger$ becomes the completion
functor $\bullet^\wedge$, while $\bullet^\dagger=(\bullet_{l.f.})^{C(e)}$.

Let $\M_{\hbar}\in \HC_{\overline{\Orb}}(\U_\hbar), \M'_\hbar\in \HC^Q_{\overline{\Orb}}(\U^\wedge_\hbar)$.
Pick a morphism $\varphi: \M_\hbar\rightarrow \left((\M'_\hbar)_{l.f.}\right)^{C(e)}$.
Compose $\varphi$ with the inclusion  $\left((\M'_\hbar)_{l.f.}\right)^{C(e)}\hookrightarrow
\M'_\hbar$. Being a homomorphism of $\U_\hbar$-modules, the corresponding map $\M_\hbar\rightarrow \M'_\hbar$
is continuous in the $I_{\chi,\hbar}$-adic topology. So it uniquely extends  to a continuous morphism
$\psi_\varphi:\M_\hbar^\wedge\rightarrow \M'_\hbar$.

On the other hand, let $\psi:\M_\hbar^\wedge\rightarrow \M'_\hbar$ be a  morphism
in $\HC^Q_{\overline{\Orb}}(\U^\wedge_\hbar)$. We have the natural homomorphism
$\M_\hbar\rightarrow \M^\wedge_\hbar$. Its image consists of $C(e)$-invariant
l.f. vectors. So the image of the composition $\M_\hbar\rightarrow \M'_\hbar$
lies in $\left((\M'_\hbar)_{l.f.}\right)^{C(e)}$. So we get
a morphism $\varphi_\psi: \M_\hbar\rightarrow \left((\M'_\hbar)_{l.f.}\right)^{C(e)}$.

It is easy to see that the assignments $\varphi\mapsto \psi_\varphi$ and $\psi\mapsto \varphi_\psi$
are inverse to each other.

(2): Let $\M^1_\hbar,\M^2_\hbar$ denote the kernel and the cokernel of the natural morphism
$\M_\hbar\rightarrow \left((\M_\hbar^\wedge)_{l.f.}\right)^{C(e)}$.
Since the completion functor is exact, we see that $\M^{1\wedge}_\hbar=0$. By Lemma \ref{Lem:3.5.2},
$\VA(\M^1_\hbar)\subset \partial\Orb$.

Let us prove that $\VA(\M^2_\hbar)\subset \partial\Orb$.

First of all, let us make a general remark.
Let $\mathcal{C}$ be an abelian category and $\mathcal{F}$ be a left exact endo-functor of $\mathcal{C}$
equipped with a functor morphism $\varphi:\operatorname{id}\rightarrow \mathcal{F}$. Consider a short
exact sequence $0\rightarrow M_1\rightarrow M\rightarrow M_2\rightarrow 0$. It gives rise to an exact sequence $\operatorname{coker}\varphi_{M_1}\rightarrow \operatorname{coker}\varphi_M\rightarrow
\operatorname{coker}\varphi_{M_2}$.

We will apply this observation to the category $\mathcal{C}$ of vector spaces equipped with compatible $\g$-,$\K^\times$- and
$Q$-actions and the functors $\mathcal{F}(\bullet)=(\bullet^\wedge)_{l.f.}^{C(e)}$ or $(\bullet^\wedge)_{\g-l.f.}^{C(e)}$.

First, consider the exact sequence $0\rightarrow \M_\hbar\xrightarrow{\hbar\cdot}\M_\hbar\rightarrow \M_\hbar/\hbar \M_\hbar\rightarrow 0$. We see that $\M^2_\hbar/\hbar \M^2_\hbar\hookrightarrow ((\M_\hbar/\hbar\M_\hbar)^\wedge)_{\g-l.f.}^{C(e)}$. So it is enough to show that the last module
is supported on $\overline{\Orb}$. But the $\K[\g^*]$-module
$\M_\hbar/\hbar \M_\hbar$ is supported on $\overline{\Orb}$ and therefore has a finite filtration,
whose successive quotients are objects of $\Coh^G(\overline{\Orb})$.


So it is enough to show that the cokernel of $M\rightarrow ((M^\wedge_\chi)_{\g-l.f.})^{C(e)}$ is supported on $\partial\Orb$ for any
$M\in \Coh^G(\overline{\Orb})$. But, according to assertion (2) of Proposition \ref{Prop:hom3},
$((M^\wedge_\chi)_{\g-l.f.})^{C(e)}=\Gamma(\Orb,M|_{\Orb})$, where $M|_{\Orb}$ stands for the restriction
of $M$ to $\Orb$. By Lemma \ref{Lem:hom1}, $\Gamma(\Orb,M|_{\Orb})$ is a finitely generated $\K[G/G_\chi]$-
(and hence $\K[\overline{\Orb}]$-) module. Repeating the argument in the first paragraph of the proof of
Proposition \ref{Prop:hom2}, we see that the restriction of $\Gamma(\Orb,M|_{\Orb})$ to $\Orb$
coincides with $M|_{\Orb}$. So the cokernel of $M\rightarrow \Gamma(\Orb, M|_{\Orb})$ is supported on $\partial\Orb$.
\end{proof}

\begin{Rem}\label{Rem_up_dag_hbar}
Of course, we can define $\N_\hbar^\dagger$ for an arbitrary (not necessarily
finite dimensional) bimodule $\N_\hbar\in \HC^Q(\Walg_\hbar)$ exactly as above.
Then $\N_\hbar^\dagger$ becomes a $\g$-l.f.  graded
$\K[\hbar]$-flat $\U_\hbar$-bimodule. However, we do not prove that $\N_\hbar^\dagger$
is finitely generated. The functor $\bullet^\dagger$ is still left exact, this can be checked
directly. Also we note that the proof of assertion (1) of Proposition \ref{Prop:3.5.5}
shows that the spaces $\Hom(\M_\hbar, \N_\hbar^\dagger)$ and $\Hom(\M_{\hbar\dagger},\N_\hbar)$
(the first $\Hom$-space is taken in the category of graded $\U_\hbar$-bimodules)
are naturally isomorphic. Using this observation and results of Ginzburg, \cite{Ginzburg},
we will see in Subsection \ref{SUBSECTION_Ginzburg} that $\N^\dagger_\hbar$ is finitely generated.
We do not use this result in the present paper.
\end{Rem}

\subsection{Construction of functors between $\HC_{\overline{\Orb}}(\U), \HC^Q_{fin}(\Walg)$}\label{SUBSECTION_construction2}
Here we construct  functors between the categories
$\HC_{\overline{\Orb}}(\U)$ and $\HC_{fin}^Q(\Walg)$.

Let $\M\in \HC(\U)$. Recall the notion of a good filtration on $\M$ introduced in Subsection \ref{SUBSECTION_completions}. We remark that
if $\F_i\M, \F'_i\M$ are two good filtrations then there are $k,l$ such that
$\F_{i-k}\M\subset \F'_i\M\subset \F_{i+l}\M$ for all $i$.
 For a good filtration
$\F_i\M$ in $\M$ set $\M_\dagger^{\F}:=R_\hbar(\M)_\dagger/(\hbar-1)R_\hbar(\M)_\dagger$
(the superscript $\F$ indicates the dependence of the bimodule on the choice of
$\F$).

Let $\varphi:\M_1\rightarrow \M_2$ be a homomorphism of two
bimodules in $\HC(\U)$. Choose good filtrations $\F^1_i\M_1,\F^2_i\M_2$.
Replacing $\F^2_i\M$  with $\F'^2_i\M^2:=\F^2_{i+k}\M^2$ for sufficiently large
$k$ we get $\varphi(\F^1_i\M_1)\subset \F^2_i\M_2$. So $\varphi$ defines
a homomorphism $\varphi_\hbar: R_\hbar(\M_1)\rightarrow R_\hbar(\M_2)$.
The homomorphism $\varphi_{\hbar\dagger}: R_\hbar(\M_1)_\dagger\rightarrow R_\hbar(\M_2)_\dagger$
gives rise to $\varphi^{\F^1,\F^2}_\dagger: \M_{1\dagger}^{\F^1}
\rightarrow \M_{2\dagger}^{\F_2}$. Clearly, if $\varphi:\M_1\rightarrow \M_2, \psi:\M_2\rightarrow \M_3$
are two homomorphisms and $\F^i$ is a good filtration on $\M_i, i=1,2,3,$ with $\varphi(\F_i^1\M_1)\subset \F^2_i\M_2,
\psi(\F_i^2\M_2)\subset \F_i^3\M_3$, then
\begin{equation}\label{eq:filtrations}
(\psi\circ\varphi)_\dagger^{\F^1,\F^3}=\psi_\dagger^{\F^2,\F^3}\circ \varphi^{\F^1,\F^2}_\dagger.
\end{equation}

For two good filtrations $\F,\F'$ on $\M$ with $\F_i\M\subset \F'_i\M$
consider the morphism $\operatorname{id}^{\F,\F'}_\dagger:\M_\dagger^{\F}\rightarrow \M_\dagger^{\F'}$.
We claim that this is an isomorphism. Indeed, for sufficiently large $k$ we have an inclusion
$\F'_{i}\M\subset \F_{i+k}\M$. Since $\operatorname{id}^{\F,\F_{\bullet+k}}_\dagger$ is the identity
morphism of $\M_\dagger^{\F}=\M_{\dagger}^{\F_{\bullet+k}}$, (\ref{eq:filtrations}) implies that
$\operatorname{id}^{\F,\F'}_\dagger$ has a left inverse. Similarly, it also has a right inverse.
Using the isomorphisms $\operatorname{id}^{\F,\F'}$ and their inverses we can identify all
$\M_\dagger^{\F}$ (the identification does not depend on the choice of intermediate filtrations
thanks to (\ref{eq:filtrations})).
Similarly, we see that
$\varphi_\dagger^{\F_1,\F_2}$ is also independent of the choice of $\F^1,\F^2$ modulo the identifications
we made.

Summarizing, we get a functor $\bullet_\dagger:\HC(\U)\rightarrow \HC^Q(\Walg)$.

Now let us construct a functor $\bullet^\dagger:
\HC_{fin}^Q(\Walg)\rightarrow \HC_{\overline{\Orb}}(\U)$. For a
module $\Nil\in \HC_{fin}^Q(\Walg)$ define a filtration $\F_i\Nil$ by
setting $\F_{-1}\Nil=\{0\}, \F_0\Nil=\Nil$. Since
$\KF_0\Walg=\KF_1\Walg=\K$ ($\K[S]$ has no component of degree 1),
we get $[\KF_i\Walg,\F_j\Nil]\subset \F_{i+j-2}\Nil$.

Put $\Nil^\dagger:=R_\hbar(\Nil)^\dagger/(\hbar-1)$. Then $\Nil^{\dagger}$ comes equipped
with a good filtration. Every
 homomorphism $\varphi:\Nil_1\rightarrow \Nil_2$ gives rise to
$\varphi^\dagger:\Nil_1^\dagger\rightarrow \Nil_2^\dagger$.
The data $\Nil\mapsto \Nil^\dagger, \varphi\mapsto \varphi^\dagger$
constitute a functor.

Interpreting  results of the previous subsection in the present
situation, we get the following proposition.

\begin{Prop}\label{Prop:3.7.1}
\begin{enumerate}
\item The functor $\bullet_\dagger: \HC(\U)\rightarrow \HC^Q(\Walg)$
is exact and maps $\HC_{\overline{\Orb}}(\U)$ to
$\HC_{fin}^Q(\Walg)$ and $\HC_{\partial\Orb}(\U)$ to zero.
\item The functor $\bullet_{\dagger}:\HC(\U)\rightarrow \HC^Q(\Walg)$ is tensor.
\item $\dim\M_\dagger=\mult_{\overline{\Orb}}\M$ for any $\M\in \HC_{\overline{\Orb}}(\U)$.
\item The functor $\bullet^\dagger:\HC^Q_{fin}(\Walg)\rightarrow \HC_{\overline{\Orb}}(\U)$ is right
adjoint to the restriction of $\bullet_\dagger$ to
$\HC_{\overline{\Orb}}(\U)$.
\item The kernel and the cokernel of the natural morphism $\M\rightarrow (\M_\dagger)^\dagger$
lie in $\HC_{\partial \Orb}(\U)$.
\end{enumerate}
\end{Prop}
\begin{proof}
Assertions (1) and (3) follow directly from Lemma \ref{Lem:3.5.2}. Also using Lemma \ref{Lem:3.5.2} one can
reduce assertion  (2) to the following claim:

Let $\A$ be a filtered algebra and $M,N$ be filtered $\A$-bimodules. Then the natural homomorphism
$R_\hbar(M)\otimes_{R_\hbar(\A)}R_\hbar(N)/(\hbar-1)(R_\hbar(M)\otimes_{R_\hbar(\A)}R_\hbar(N))\rightarrow M\otimes_A N$
is an isomorphism.

The claim follows from the observation that $R_\hbar(M)\otimes_{R_\hbar(\A)}R_\hbar(N)/(\hbar-1)(R_\hbar(M)\otimes_{R_\hbar(\A)}R_\hbar(N))$ is naturally
identified with $\A\otimes_{R_\hbar(\A)}R_\hbar(M)\otimes_{R_\hbar(\A)}R_\hbar(N)\otimes_{R_\hbar(\A)}\A$.
But the latter is nothing else but $M\otimes_{R_\hbar(\A)}N=M\otimes_\A N$.

Let us derive assertion (4) from the corresponding assertion of Proposition \ref{Prop:3.5.5}.
Pick $\varphi\in \Hom(\M,\N^\dagger)$, where $\M\in \HC_{\overline{\Orb}}(\U), \N\in \HC^Q_{fin}(\Walg)$.
There are good filtrations $\F_i\M, \F'_i\N$ on $\M,\N$ respectively such that $\varphi(\F_i\M)\subset \F'_i\N^\dagger$
(here in the right hand side $\F'_i\N^\dagger$ is a good filtration arising from $\F'_i\N$, see the construction
above). So $\varphi$ gives rise to a morphism $\varphi_\hbar: R_\hbar(\M)\rightarrow R_\hbar(\N^\dagger)=(R_\hbar(\N))^\dagger$ in $\HC(\U_\hbar)$. The corresponding homomorphism
$\psi_\hbar: R_\hbar(\M)_\dagger\rightarrow R_\hbar(\N)$ gives rise to $\psi: \M_\dagger\rightarrow \N$.
Similarly to the construction of the functors above, the morphism $\psi$ does not depend on the choice of good filtrations. So we get a natural map $\Hom(\M,\N^\dagger)\rightarrow \Hom(\M_\dagger,\N)$.  The inverse map
is constructed in a similar way.

Assertion (5) follows now directly from assertion (2) of Proposition \ref{Prop:3.5.5}.
\end{proof}

\begin{Rem}\label{Rem:3.5.6}
It follows easily from the construction that for $\J\in
\Id_\Orb(\U)$ the definition of $\J_\dagger$ given here is the same
as in Subsection \ref{SUBSECTION_corr_ideals}.
\end{Rem}

\begin{Rem}\label{Rem_up_dag}
One can define $\N^\dagger$ for an arbitrary object $\N\in \HC^Q(\Walg)$ using the same procedure as above
and Remark \ref{Rem_up_dag_hbar}. We still have the natural isomorphism $\Hom(\M, \N^\dagger)\cong
\Hom(\M_\dagger,\N)$ and hence the natural morphism $\M\rightarrow (\M_\dagger)^\dagger$.
The functor $\bullet^\dagger$ is left exact.

In particular, for an ideal $\I\subset \Walg$ we have defined the ideal $\I^\dagger\subset \U$,
see Subsection \ref{SUBSECTION_corr_ideals}. One can show that the {\it functor} $\bullet^\dagger$
maps $\Walg$ to $\U$. And so for a $C(e)$-stable ideal $\I$ both definitions of $\I^\dagger$
agree. We are not going to use this. Instead we observe the following. By the definition of $\I^\dagger\subset \U$
given in Subsection \ref{SUBSECTION_corr_ideals}, $\I$ is nothing else but the preimage of
$\I^\dagger\subset \Walg^\dagger$ under the natural homomorphism $\U\rightarrow \Walg^\dagger=(\U_\dagger)^\dagger$
provided $\I$ is $Q$-stable.
In particular, it follows that $\I^\dagger\subset \U$ is the kernel of the natural map $\U\rightarrow \Walg^\dagger/\I^\dagger\hookrightarrow (\Walg/\I)^\dagger$.
\end{Rem}

\begin{Rem}\label{Rem:4.5.1}
We can use the previous remark and results of Borho and Kraft, \cite{BoKr}, to prove the
Joseph irreducibility theorem, \cite{Joseph}, see the discussion at the end of Subsection \ref{SUBSECTION_corr_ideals}.
Other proofs can be found
in \cite{Vogan},\cite{Ginzburg_irr}.

If $\I$ is a $C(e)$-stable ideal of finite codimension in $\Walg$, then, thanks
to the previous remark, $\VA(\U/\I^\dagger)\subset \VA((\Walg/\I)^\dagger)\subset\overline{\Orb}$.
On the other hand, $(\I^\dagger)_\dagger\subset \I$  by assertion (ii) of Theorem \ref{Thm:5}. So $\VA(\U/\I^\dagger)=\overline{\Orb}$.

Now we are ready to rederive the Joseph irreducibility theorem.
Suppose $\J$ is a primitive ideal in $\U$ such that $\overline{\Orb}$ is an irreducible component
of $\VA(\U/\J)$ of maximal dimension. Therefore $\J_\dagger$ has finite codimension in $\Walg$. By assertion (ii)
of Theorem \ref{Thm:5}, $\J\subset (\J_\dagger)^\dagger$.
Now  Corollar 3.6, \cite{BoKr},
implies that $\J=(\J_\dagger)^{\dagger}$. So $\VA(\U/\J)=\overline{\Orb}$.
\end{Rem}

To finish the subsection we will prove a straightforward generalization of
\cite{Wquant}, Proposition 3.4.6, which was conjectured by McGovern
in \cite{McGovern}.

\begin{Prop}\label{Prop:3.7.3}
Let $\A$ be a Dixmier algebra (i.e. an algebra over $\U$ that is a
Harish-Chandra bimodule with respect to the  left and right multiplications by elements of
$\U$) such that $\VA(\A)=\overline{\O}$. Suppose, in addition, that $\A$
is prime. Then $\Goldie(\A)\leqslant \sqrt{\mult_{\overline{\Orb}}(\A)}$.
\end{Prop}
\begin{proof}
First of all, let us show that $\A$ admits a good filtration that is also an algebra filtration.
Choose an $\ad \g$-stable subspace $\A_2\subset \A$ that contains the image of $\g$ in $\A$
and generates $\A$ as a left  $\U$-module. Set $\A_0:=\A_1=\K, \A_k:=\F_{k-2}\U\cdot\A_2$
for $k\geqslant 2$. The filtration $\A_k$ on $\A$ is a good filtration but not an algebra filtration, in general.
But there is $m\geqslant 0$ with $\A_2^2\subset \A_{4+m}$.  Set $\F_0\A=\ldots =\F_{m+1}\A:=\K, \F_k\A:=\A_{k-m}$
for $k\geqslant m+2$. Then $\F_k\A$ is a good filtration, and $\F_k\A\F_l\A=\F_k\A$ if $l\leqslant m+1$ and $\F_{k}\A\F_{l}\A=\F_{k-m-2}\U \A_2\F_{l-m-2}\U \A_2\subset \F_{k+l-2m-4}\U\A_{4+m}=\A_{k+l-m}=\F_{k+l}\A$.
So $\F_\bullet\A$ is a good algebra filtration.

From the construction we see that $\A_\dagger, (\A_\dagger)^\dagger$ have  natural algebra structures
and the natural morphism $\A\rightarrow (\A_\dagger)^\dagger$ is a homomorphism of algebras.
Analogously to the proof of Proposition 3.4.6 in \cite{Wquant}, we
have a homomorphism $\psi:\A\rightarrow \B\otimes \A_\dagger$, where
$\B$ is a certain completely prime (=without zero divisors) algebra.  Let us recall the construction of
$\B$ and $\psi$.

By definition, $\B:=(\W_\hbar^\wedge)_{\K^\times-l.f}/(\hbar-1)$. Set $\A_\hbar:=R_\hbar(\A)$.
Since $\A_{\hbar\dagger}$ has finite rank over $\K[\hbar]$ we see that $(\A_\hbar^\wedge)_{\K^\times-l.f.}=
(\W_\hbar^\wedge\widehat{\otimes}_{\K[[\hbar]]}\A_{\hbar\dagger}^\wedge)_{\K^\times-l.f.}=
(\W_{\hbar}^\wedge)_{\K^\times-l.f.}\otimes_{\K[\hbar]}\A_{\hbar\dagger}$ and so $(\A_\hbar^{\wedge})_{\K^\times-l.f.}/
(\hbar-1)(\A_\hbar^{\wedge})_{\K^\times-l.f.}=\B\otimes \A_\dagger$. Now $\psi$ is obtained from the natural
homomorphism $\A_\hbar\rightarrow (\A_\hbar^\wedge)_{\K^\times-l.f.}$.

Similarly to Subsection \ref{SUBSECTION_corr_ideals}, for an ideal $\I\subset \A_\dagger$ we can define the ideal
$\I^\dagger\subset\A$.  The construction implies  $\I^\dagger=\psi^{-1}(\B\otimes\I)$ and
$(\I^\dagger)_\dagger\subset \I$.

Let $\I$ be a minimal
prime ideal of $0$ in $\A_\dagger$. Set $\J:=\psi^{-1}(\B\otimes\I)$ in $\A$.
We are going to show that $\J=\{0\}$.

Assume the converse.  Since the algebra $\A$ is prime, we can apply results of Borho and Kraft,
\cite{BoKr}, to see that $\A/\J$ is supported on $\partial{\Orb}$,
equivalently, $\A_\dagger=\J_\dagger$. However
$\J_\dagger\subset \I$, contradiction. So we have an embedding $\A\hookrightarrow \B\otimes (\A_\dagger/\I)$.

Now, similarly to
\cite{Wquant}, $\Goldie(\A)\leqslant
\Goldie(\B\otimes(\A_\dagger/\I))=\Goldie(\A_\dagger/\I)
\leqslant\sqrt{\dim\A_\dagger}=\sqrt{\mult_{\Orb}(\A)}$.
\end{proof}

\subsection{Comparison with Ginzburg's construction}\label{SUBSECTION_Ginzburg}
Ginzburg, \cite{Ginzburg}, defined a functor
$\HC_\Orb(\U)\rightarrow \HC_{fin}^Q(\Walg)$ in the following
way: $\M\mapsto (\M/\M\m_{\chi})^{\ad \m}$ (to see the action of
$Q$ one needs to prove that the natural homomorphism
$(\M/\g_{\leqslant -2,\chi})^{\ad\g_{\leqslant
-1}}\rightarrow(\M/\m_\chi\M)^{\ad \m}$ is an isomorphism, this can
be done  similarly to \cite{GG}, Subsection 5.5). Below in this subsection we will check that
 Ginzburg's functor coincides with ours. In particular, on the language of the
quantum Hamiltonian reduction one has
$\J_\dagger=(\J/\J\m_\chi)^{\ad\m}$.

Recall the algebras $\U^\heartsuit:=(\U_\hbar^\wedge)_{\K^\times-l.f.}/(\hbar-1),
\W(\Walg)^\heartsuit:=(\W^\wedge_\hbar(\Walg^\wedge_\hbar))_{\K^\times-l.f.}/(\hbar-1)$ introduced
in \cite{Wquant}. Let $\Phi:\U^\heartsuit\rightarrow \W(\Walg)^\heartsuit$ be the isomorphism induced by $\Phi_\hbar$.
Now let $\M$ be a Harish-Chandra $\U$-bimodule. Choosing a good filtration on $\M$, we get a Harish-Chandra $\U_\hbar$-bimodule $\M_\hbar=R_\hbar(\M)$.
Set $\M^\heartsuit:=(\M^\wedge_\hbar)_{\K^\times-l.f.}/(\hbar-1)(\M^\wedge_\hbar)_{\K^\times-l.f.}$.
Analogously to the previous subsection,
the $\U^\heartsuit$-bimodule $\M^\heartsuit$ does not depend on the choice of a filtration on $\M$.
Moreover, from the construction of $\M_\dagger$ it follows that $\M^\heartsuit= \W(\M_\dagger)^\heartsuit(:=\W_\hbar^\wedge(\M_{\hbar\dagger})_{\K^\times-l.f.}/
(\hbar-1)\W_\hbar^\wedge(\M_{\hbar\dagger})_{\K^\times-l.f.})$.
So the $\Walg$-bimodule $\M_\dagger$ is nothing else but $(\M^\heartsuit)^{\ad V}\cong (\M^\heartsuit/\M^\heartsuit\m)^{\ad\m}$, where we consider $\m$ as a lagrangian subspace in $V$.

The embedding $\U\hookrightarrow \U^\heartsuit$ gives rise to a map $\U/\U\m_\chi\rightarrow \U^\heartsuit/\U^\heartsuit \m_\chi$. As we have seen in \cite{Wquant}, the paragraph preceding Remark 3.2.7, $\U^\heartsuit=\U+\U^\heartsuit\m_\chi$. Also it is clear from the construction there that
$\U\cap \U^\heartsuit\m_\chi=\U\m_\chi$ (this was used implicitly in the proof of Corollary 3.3.3 in \cite{Wquant}).
So the natural homomorphism $(\U/\U\m_\chi)^{\ad\m_\chi}\rightarrow (\U^\heartsuit/\U^\heartsuit\m_\chi)^{\ad\m_\chi}=\Walg$ is an isomorphism of filtered algebras.

So we have a functorial homomorphism $\iota:(\M/\M\m_\chi)^{\ad\m_\chi}\rightarrow (\M^\heartsuit/\M^\heartsuit\m_\chi)^{\ad\m_\chi}=\M_{\dagger}$ of $\Walg$-modules.
This homomorphism preserves  natural (Kazhdan) fintrations on the bimodules.
Let us check that $\iota$ is an isomorphism. By assertion (3) of Lemma \ref{Lem:3.5.2}  and part (i) of Theorem 4.1.4 in \cite{Ginzburg}, we have $\gr (\M/\M\m_\chi)^{\ad\m}\cong \gr\M|_{S}\cong \gr\M_\dagger$. Moreover,
the corresponding isomorphism $\gr (\M/\M\chi)^{\ad\m}\rightarrow\gr\M_\dagger$ coincides with
$\gr\iota$. Since the gradings on both modules are bounded from below, we see that
$\iota$ is an isomorphism.

\begin{Rem}
As Ginzburg proved in \cite{Ginzburg}, Theorem 4.2.2, there is a right adjoint functor to
$\bullet_\dagger: \HC(\U)\rightarrow \HC(\Walg)$ (he did not considered $Q$-equivariant
structures). We can consider the functor $\bullet^{\widetilde{\dagger}}$ from $\HC(\Walg)$
to the category of $\g$-l.f. $\U$-bimodules corresponding to taking l.f.
sections on the homogeneous level (without taking $C(e)$-invariants).
Similarly to Remark \ref{Rem_up_dag}, we see that
$\Hom(\M, \N^{\widetilde{\dagger}})=\Hom(\M_\dagger, \N)$. Thanks to Ginzburg's result,
the image of $\bullet^{\widetilde{\dagger}}$ lies in $\HC(\U)$. In particular, we see
that the image of $\bullet^\dagger$ lies in $\HC(\U)$. We are not going to use this result below.
\end{Rem}

\section{Proofs of Theorems \ref{Thm:1},\ref{Thm:2}}\label{SECTION_proofs}
\subsection{Surjectivity theorem}\label{SUBSECTION_Surjectivity}
The following theorem will be used to  complete the proof of Theorem \ref{Thm:1}
and  implies the most non-trivial part of Theorem \ref{Thm:2},
assertion 5.

\begin{Thm}\label{Thm:surjectivity}
Let $\M\in \HC(\U)$ and let $\N\subset \M_\dagger$ be a $Q$-stable subbimodule
of finite codimension. Let $\N^{\ddag}$ stand for the preimage of $\N^\dagger\subset
(\M_\dagger)^\dagger$ in $\M$ (see Remark \ref{Rem_up_dag}).
Then $(\N^\ddag)_\dagger=\N$.
\end{Thm}
The ideas behind the proof were briefly explained in Subsection \ref{SSS_4}.
\begin{proof}
Examining the construction of the functors we see that the assertion of the theorem stems from the following claim:
\begin{itemize}
\item[(*)] Let $\M_\hbar\in \HC(\U_\hbar)$ and $\N'_\hbar$ be a $\g$- (with respect to the action
$\xi\mapsto \frac{1}{\hbar^2}[\xi,\cdot]$), $\K^\times$- and $Q$-stable
(but not necessary $\hbar$-saturated) $\U^\wedge_\hbar$-subbimodule in $\M^\wedge_\hbar$.
Then $\N_\hbar'$ is the completion of its preimage $\N_\hbar$ in $\M_\hbar$.
\end{itemize}

For a $\U_\hbar^\wedge$-subbimodule $\N_\hbar'\subset \M^\wedge_\hbar$ we define its
$\hbar$-saturation $\widehat{\N}'_\hbar$ as the subset of $\M^\wedge_\hbar$ consisting of all
elements $m\in \M_\hbar^\wedge$ with $\hbar^k m\in \N_\hbar'$. Since $\M^\wedge_\hbar$ is Noetherian
there is $N\in\NN$ with $\hbar^N \widehat{\N}'_\hbar\subset \N'_\hbar$.

Let us show that if (*) holds for $\widehat{\N}'_\hbar$, then it holds for $\N'_\hbar$.
So suppose that the completion $\widehat{\N}_\hbar^\wedge$ of the preimage $\widehat{\N}_\hbar$ of $\widehat{\N}'_\hbar$
coincides with $\widehat{\N}'_\hbar$.

 Consider the subbimodule $\N'_\hbar/\N^\wedge_\hbar\subset (\widehat{\N}_\hbar/\N_\hbar)^{\wedge}$. We remark
that $\widehat{\N}_\hbar/\N_\hbar$ is embedded into $\widehat{\N}'_\hbar/\N'_\hbar$  and hence
is annihilated by some power of $\hbar$ and
 is supported on $\overline{\Orb}$. Let $I_\hbar(\Orb)$ denote the preimage of
$I(\Orb)$ in $\U_\hbar$. We see that some power of $I_\hbar(\Orb)$ annihilates
$\widehat{\N}_\hbar/\N_\hbar$. Let $M$ denote the (no matter, left or right) annihilator
of $I_\hbar(\Orb)$ in $\widehat{\N}_\hbar/\N_\hbar$. By Corollary \ref{Cor:3.0.1}, the annihilator
of $I_\hbar(\Orb)$ in $(\widehat{\N}_\hbar/\N_\hbar)^\wedge$ coincides with $M^\wedge$. In particular,
$N':=\N'_\hbar\cap M^\wedge\neq \{0\}$. Both $M^\wedge$ and $N'$ are objects in $\HVB^\wedge_{G/G_\chi}$.
It follows from Propositions \ref{Prop:hom2},\ref{Prop:hom3} that  $N'$ coincides
with the completion of its preimage $N$ in $M$. Consider the preimage $\widetilde{\N}_\hbar$
of $N$ under the projection $\widehat{\N}_\hbar\rightarrow \widehat{\N}_\hbar/\N_\hbar$. Then $\widetilde{\N}_\hbar^\wedge\subset \N'_\hbar$. Therefore $\widetilde{\N}_\hbar=\N_\hbar$. Contradiction.
So we have proved that (*) for $\widehat{\N}_\hbar$ implies (*) for $\N_\hbar$.

In particular, we see that (*) holds for $\N'_\hbar$ provided $\hbar^k \M^\wedge_\hbar\subset \N'_\hbar$.
Also we see that it is enough to prove (*) for $\hbar$-saturated $\N'_\hbar$. Below $\N'_\hbar$ is
assumed to be $\hbar$-saturated.

Set $\N'_{\hbar,k}=\N'_\hbar+\hbar^{k+1} \M^\wedge_\hbar$. Let
$\N_{\hbar,k}$ denote the preimage of $\N'_{\hbar,k}$ in $\M_\hbar$. By the above, $\N^\wedge_{\hbar,k}=\N'_{\hbar,k}$.
On the other hand, $\bigcap_k \N'_{\hbar,k}=\N'_\hbar$ because $\N'_\hbar$ is closed in
$\M^\wedge_\hbar$, see Lemma \ref{Lem:1.231}. Therefore $\bigcap_k \N_{\hbar,k}=\N_\hbar$.
Since $\N'_\hbar \subset \M^\wedge_\hbar$ is $\hbar$-saturated, we see that so is  $\N_\hbar\subset \M_\hbar$.
Replace $\M_\hbar$ with $\M_\hbar/\N_\hbar$ and $\N'_\hbar$ with $\N'_\hbar/\N^\wedge_\hbar$.
So we may and will assume that $\M_\hbar\in \HC_{\overline{\Orb}}(\U_\hbar)$ and $\N_\hbar=\{0\}$. We need to check that $\N'_\hbar=\{0\}$. Assume the converse.

Set $T_k:=\N_{\hbar,k}/\N_{\hbar,k+1}$. By definition, this is a
$\U_{\hbar}/(\hbar^{k+2})$-module. However,
it is easy to see that \begin{equation}\label{eq:4.3:1}
\hbar\N_{\hbar,k}\subset \N_{\hbar,k+1},
\end{equation}
 So $\hbar$ acts trivially
on $T_k$ and $T_k$ is a $\K[\g^*]$-module. Moreover,
(\ref{eq:4.3:1}) implies that the multiplication by $\hbar$ induces
a homomorphism $T_k\rightarrow T_{k+1}$ of $\K[\g^*]$-modules also
denoted by $\hbar$. So $T:=\bigoplus_{i=0}^\infty T_i$ becomes a
$\K[\g^*][\hbar]$-module.

Suppose for a moment that $T$ is a finitely generated $\K[\g^*][\hbar]$-module.
 It follows that there is $k>0$ such that $T_i=\hbar^{i-k}T_k$ for all $i>k$. This implies
\begin{equation}\label{eq:surj1}\N_{\hbar,i}=\hbar^{i-k}\N_{\hbar,k}+
\N_{\hbar,i+1}.\end{equation}
Now recall that $\M_\hbar$ is graded, the grading is bounded from below, and all graded components
are finite dimensional. All $\N_{\hbar,i}$ are graded sub-bimodules in $\M_\hbar$.
(\ref{eq:surj1}) implies that for any $k$ the $k$-th graded component of $\N_\hbar$
coincides with that of $\N_{\hbar,i}$ for sufficiently large $i$. Also (\ref{eq:surj1})
implies that the inverse sequence of the projections of $\N_{\hbar,i}$ to $\M_{\hbar,0}$
stabilizes. So we can find $k$ such that the $k$-th graded component of $\N_{\hbar,i}$
is nonzero for all $i$. It follows that $\N_\hbar\neq \{0\}$. Contradiction.

To prove that $T$ is finitely generated we use the following construction.
Set  $C:=[\M^\wedge_\hbar/\N'_{\hbar,0}]_{l.f.}$. The argument of  Lemma \ref{Lem:3.5.1}
implies that $C$ is a finitely generated $\K[\g^*]$-module. The following lemma shows that $T$ is a submodule
in $C[\hbar]$ and hence is finitely generated. This completes the proof of the theorem.
\end{proof}

\begin{Lem}\label{Lem:4.3.2}
There is an embedding $T\hookrightarrow C[\hbar]$ of
$\K[\g^*][\hbar]$-modules.
\end{Lem}
\begin{proof}
We will construct embeddings $\iota_i:T_i\hookrightarrow
C,i=0,1,\ldots,$ such that $\iota_{i+1}(\hbar x)=\iota_i(x)$ for all $x\in T_i$.
For an embedding $T\subset C[\hbar]$ we will take the direct sum of $\iota_i$'s.

Since $\M^\wedge_\hbar/\N'_\hbar$ is $\K[\hbar]$-flat, we have the following exact sequence
\begin{equation}
0\rightarrow \M^\wedge_\hbar/\N'_{\hbar,0}\rightarrow
\M^\wedge_{\hbar}/\N'_{\hbar,k+1}\rightarrow
\M^\wedge_{\hbar}/\N'_{\hbar,k}\rightarrow 0,
\end{equation}
where the first map is the multiplication by  $\hbar^{k+1}$, and thus an exact sequence
\begin{equation}
0\rightarrow
C\rightarrow
(\M^\wedge_{\hbar}/\N'_{\hbar,k+1})_{l.f.}\rightarrow
(\M^\wedge_{\hbar}/\N'_{\hbar,k})_{l.f.}
\end{equation}
There is a natural inclusion
$$T_k\hookrightarrow \M_\hbar/\N_{\hbar,k+1}\hookrightarrow
(\M^\wedge_{\hbar}/\N'_{\hbar,k+1})_{l.f.},$$
whose image in
$(\M^\wedge_{\hbar,k}/\N'_{\hbar,k})_{l.f.}$ is zero.
So we get a $\K[\g^*]$-module embedding
$T_k\hookrightarrow C$.

The claim that these embeddings are compatible with the
multiplication by $\hbar$ stems from the following commutative
diagram.

\begin{picture}(120,70)
\put(15,62){$0$}\put(15,42){$C$}
\put(1,22){$(\M^\wedge_{\hbar}/\N'_{\hbar,k+1})_{l.f.}$}
\put(3,2){$(\M^\wedge_{\hbar}/\N'_{\hbar,k})_{l.f.}$}
\put(95,62){$0$}\put(95,42){$C$}
\put(81,22){$(\M^\wedge_{\hbar}/\N'_{\hbar,k+2})_{l.f.}$}
\put(81,2){$(\M^\wedge_{\hbar}/\N'_{\hbar,k+1})_{l.f.}$}
\put(36,35){$T_k$} \put(70,35){$T_{k+1}$}
\put(16,60){\vector(0,-1){12}} \put(16,40){\vector(0,-1){12}}
\put(16,20){\vector(0,-1){12}} \put(96,60){\vector(0,-1){12}}
\put(96,40){\vector(0,-1){12}} \put(96,20){\vector(0,-1){12}}
\put(20,44){\vector(1,0){73}} \put(40,37){\vector(1,0){28}}
\put(35,24){\vector(1,0){45}} \put(33,4){\vector(1,0){47}}
\put(55,45){\tiny $\operatorname{id}$}
\put(55,38){\tiny $\hbar$}
\put(55,25){\tiny $\hbar$}
\put(55,5){\tiny $\hbar$}
\put(35,34){\vector(-1,-1){8}} \put(77,34){\vector(1,-1){8}}
\put(35,37){\vector(-3,1){16}} \put(77,37){\vector(3,1){16}}
\end{picture}
\end{proof}

\subsection{Completing the proofs}\label{SUBSECTION_finish}
Below for $\I\subset \Walg$ the notation $\I^\dagger$ means an ideal in $\U$ (so that we follow
the conventions of \cite{Wquant} and of Subsection \ref{SUBSECTION_corr_ideals}).

Theorem \ref{Thm:1} follows directly from Theorem \ref{Thm:surjectivity} with $\M=\U$.

\begin{proof}[Proof of Conjecture \ref{Conj:0}]
Thanks to \cite{Wquant}, Theorem 1.2.2(viii),
we need to prove that $Q$ acts transitively on the set of minimal
prime ideals $\I_1,\ldots,\I_l$ of $\J_\dagger$, where $\J\in
\Id_{\overline{\Orb}}(\U)$ is primitive. The ideal $\cap_{\gamma\in
C(e)}\gamma \I_1$ is $Q$-stable and so, by Theorem \ref{Thm:1},
 $\J^1_\dagger=\cap_{\gamma\in
C(e)}\gamma \I_1$, where $\J^1:=(\cap_{\gamma\in C(e)}\gamma
\I_1)^\dagger$. But $\J=\I_1^\dagger\supset \J^1\supset (\bigcap_{i=1}^l \I_i)^\dagger=\J$.
We deduce that $\J_\dagger=\bigcap_{\gamma\in C(e)}\gamma\I_1=\bigcap_{i=1}^l \I_i$.
Since $\Walg/(\bigcap_{i=1}^l\I_i)\cong \bigoplus_{i=1}^l \Walg/\I_i$, we see that any $\I_i$ has the form $\gamma \I_1$. 
\end{proof}

\begin{proof}[Proof  of
Theorem \ref{Thm:2}]
Assertions (1),(2),(3) follow from Proposition \ref{Prop:3.5.5}.

Let us check assertion (4) for the left annihilators (right ones are completely analogous).
Set $\J:=\LAnn_\U(\M), \I:=\LAnn_{\Walg}(\M_\dagger)$.
Since $\bullet_\dagger$ is an exact tensor functor, we see that $\J_\dagger\subset \I$. On the other hand,
by Theorem \ref{Thm:1}, $\I=\widetilde{\J}_\dagger$ for $\widetilde{\J}=\I^\dagger$.
Again, since $\bullet_\dagger$ is a tensor functor, we see that $(\widetilde{\J}\M)_\dagger=\I\M_\dagger=0$.
So $\widetilde{\J}\M\in \HC_{\partial\Orb}(\U)$. Set $\J^1:=\LAnn_{\U}(\widetilde{\J}\M)$.
We have $\J^1\widetilde{\J}\subset \J$. So $\J^1_\dagger\I=(\J^1\widetilde{\J})_\dagger\subset \J_\dagger$.
But $\J^1_\dagger=\Walg$ so $\I\subset \J_\dagger$.

Proceed to the proof of (5). By assertion (1) of Proposition \ref{Prop:3.5.5},
$\bullet_\dagger$ descends to $\HC_\Orb(\U)$. Abusing the notation we write $\bullet^\dagger$ for the composition
of $\bullet^\dagger:\HC^Q_{fin}(\Walg)\rightarrow \HC_{\overline{\Orb}}(\U)$
and the projection $\HC_{\overline{\Orb}}(\U)\rightarrow \HC_{\Orb}(\U)$. This functor
is right adjoint to $\bullet_\dagger:\HC_\Orb(\U)\rightarrow \HC_{fin}^Q(\Walg)$.
By assertion (5) of Proposition
\ref{Prop:3.5.5}, $\bullet^\dagger$ is left inverse to $\bullet_\dagger$.
From here using some abstract nonsense we see that $\bullet_\dagger$ is an equivalence
onto its image.

The claim that the image of $\bullet_\dagger$ is closed under taking subobjects follows from
Theorem \ref{Thm:surjectivity}. Then the image is automatically closed with respect to taking
subquotients.

\end{proof}

\begin{proof}[Proof of Corollary \ref{Cor:3}]
Thanks to assertion 5 of Theorem \ref{Thm:2}, it is enough to show that $\M_\dagger$ is completely reducible
in $\HC_{fin}^Q(\Walg)$. By assertion 4 and Theorem \ref{Thm:1} (together with the proof of Conjecture \ref{Conj:0})
the left and right annihilators of $\M_{\dagger}$ are intersections of primitive ideals of
finite codimension. So $\Walg$ acts on $\M_{\dagger}$ via an epimorphism to the direct sums of matrix algebras.
Therefore $\M_\dagger$ is completely reducible.
\end{proof}

\end{document}